\documentclass[12pt]{amsart}
\usepackage{amssymb,amsmath,amsthm,mathrsfs,multirow,xcolor,framed,url,upgreek}
\usepackage[pdftex,
         pdftitle={Transformation properties of Andrews-Beck $NT$ functions and generalized Appell-Lerch series},
         pdfproducer={Latex with hyperref},
         pdfcreator={pdflatex}]{hyperref}
\newtheorem{theorem}{Theorem}[section]
\newtheorem{lemma}[theorem]{Lemma}
\newtheorem{cor}[theorem]{Corollary}

\newtheorem{prop}[theorem]{Proposition}

\theoremstyle{definition}

\theoremstyle{remark}
\newtheorem{remark}[theorem]{Remark}

\numberwithin{equation}{section}

\newcommand\nutwid{\overset {\text{\lower 3pt\hbox{$\sim$}}}\nu}

\newcommand{\abs}[1]{\lvert#1\rvert}

\newcommand{\uhp}{\mathscr{H}}  

\newcommand{\sgn}{\textnormal{sgn}}

\newcommand\leg[2]{\genfrac{(}{)}{}{}{#1}{#2}} 

\newcommand\abcdMAT{\begin{pmatrix} a & b \\ c & d \end{pmatrix}}

\newcommand\MAT[4]{\begin{pmatrix} {#1} & {#2} \\ {#3}  & {#4} \end{pmatrix}}
\newcommand\Z{\mathbb{Z}}
\newcommand\Q{\mathbb{Q}}
\newcommand\R{\mathbb{R}}
\newcommand\C{\mathbb{C}}
\allowdisplaybreaks

\newcommand{\beqs}{\begin{equation*}}
\newcommand{\eeqs}{\end{equation*}}
\newcommand{\beq}{\begin{equation}}
\newcommand{\eeq}{\end{equation}}

\DeclareMathOperator{\IM}{Im}
\DeclareMathOperator{\RE}{Re}

\DeclareMathOperator{\ord}{ord}

\newcommand\rmi{i}
\newcommand\rme{e}
\newcommand\slZ{\mathrm{SL}_2(\mathbb{Z})}
\newcommand\tbtmat[4]{\left(\begin{smallmatrix}{#1} & {#2} \\ {#3} & {#4}\end{smallmatrix}\right)}

\renewcommand\uppi\pi

\makeatletter
\@namedef{subjclassname@2020}{\textup{2020} Mathematics Subject Classification}
\makeatother

\begin{document}

\title[Transformation properties of Andrews-Beck $NT$ functions]{Transformation properties of Andrews-Beck $NT$ functions and generalized Appell-Lerch series}

\author{Rong Chen}
\address{Department of Mathematics, Shanghai Normal University, Shanghai, People's Republic of China}
\email{rchen@shnu.edu.cn}

\author{Xiao-Jie Zhu}
\address{School of Mathematical Sciences,
Key Laboratory of MEA(Ministry of Education) \& Shanghai Key Laboratory of PMMP,
East China Normal University}
\email{zhuxiaojiemath@outlook.com}



\begin{abstract}
In 2021, Andrews mentioned that George Beck introduced a partition statistic $NT(r,m,n)$ which is related to Dyson's rank statistic. Motivated by Andrews's work, scholars have established a number of congruences and identities involving $NT(r,m,n)$. In this paper, we strengthen and extend a recent work of Mao on the transformation properties of the $NT$ function and provide an analogy of Hickerson and Mortenson's work on the rank function. As an application, we demonstrate how one can deduce from our results many identities involving $NT(r,m,n)$ and another crank-analog statistic $M_\omega(r,m,n)$. As a related result, some new properties of generalized Appell-Lerch series are given.
\end{abstract}

\maketitle

\section{Introduction}
\subsection{Background and the main results}
A partition of a positive integer $n$ is defined as a sequence of positive integers in non-increasing order that sums to $n$. The number of all partitions of $n$ is denoted by $p(n)$. The following distinguished congruences were discovered by Ramanujan in \cite{Ra-19}:
\begin{align}
p(5n+4)\equiv 0 &\pmod 5,\label{1.1}\\
p(7n+5)\equiv 0 &\pmod 7,\label{1.2}\\
p(11n+6)\equiv 0 &\pmod {11}\label{1.3},
\end{align}
where $n$ is any nonnegative integer.

In an effort to provide purely combinatorial interpretations of Ramanujan's famous congruences \eqref{1.1}-\eqref{1.3}, two important partition statistics of ordinary partitions, rank and crank, were introduced by Dyson, Andrews, and Garvan. In 1944, Dyson \cite{Dyson-1944} defined the rank of a partition to be the largest part of the partition minus the number of parts. Dyson also conjectured that for $k=5,7$,
$$
N(r,k,kn-s_k)=\frac{p(kn-s_k)}{k},
$$
where $N(r,k,n)$ counts the number of partitions of $n$ with rank congruent to $r$ modulo $k$ and $s_k$ denotes $s_k=(k^2-1)/24$ throughout the paper. In 1954, Dyson's conjectures were proved by Atkin and Swinnerton-Dyer \cite{Swinnerton-Dyer-1954}. Therefore, the ranks of partitions provide purely combinatorial descriptions of Ramanujan's congruences \eqref{1.1} and \eqref{1.2}. Unfortunately, Dyson's rank failed to account for Ramanujan's third congruence \eqref{1.3} combinatorially. Dyson conjectured the existence of an unknown partition statistic, which he whimsically called ``the crank'', to explain Ramanujan's third congruence modulo 11. The crank was found by Andrews and Garvan \cite{Andrews-1988} who defined it as the largest part, if the partition has no ones, and otherwise as the difference between the number of parts larger than the number of ones and the number of ones.
In 1987, Garvan \cite{Garvan-1988} proved that for $k=5,7,11$,
$$
M(r,k,kn-s_k)=\frac{p(kn-s_k)}{k},
$$
where $M(r,k,n)$ counts the number of partitions of $n$ with crank congruent to $r$ modulo $k$. Garvan's results imply that the crank of partitions provides purely combinatorial descriptions of Ramanujan's three congruences \eqref{1.1}-\eqref{1.3}.

The transformation properties of the rank function was studied by Bringmann, Ono and Rhoades \cite{BrOnRh-08} and later by many authors such as Ahlgren, Garvan, Hickerson, Mortenson and Treneer \cite{AhTr-08}, \cite{Ga-19}, \cite{HiMo-17}. For example, let
$$
D(a,M):=\sum_{n=0}^\infty\left(N(a,M,n)-\frac{p(n)}{M}\right)q^n.
$$
Hickerson and Mortenson \cite[Theorem 4.1]{HiMo-17} state that
\beq
\label{hm}
D(a,M)=d(a,M)+T_{a,M},
\eeq
where $d(a,M)$ are some given Appell-Lerch series (under $M$-dissection when $\gcd(M,6)=1$) and $T_{a,M}$ is a theta function.

In 2021, Andrews \cite{Andrews-2021} mentioned that George Beck had introduced the partition statistics $NT(r,m,n)$ and $M_\omega(r,m,n)$, which count the total number of parts in the partitions of $n$ with rank congruent to $r$ modulo $m$, and the total number of ones in the partitions of $n$ with crank congruent to $r$ modulo $m$, respectively. Namely,
\begin{align*}
NT(r,m,n)=\sum_{\substack {\pi\vdash n,\\ \text{rank}(\pi)\equiv r \pmod m}} \sharp(\pi)
\end{align*}
and
\begin{align*}
M_{\omega}(r,m,n)=\sum_{\substack{\pi\vdash n,\\ \text{crank}(\pi)\equiv r \pmod m}} \omega(\pi),
\end{align*}
where $\sharp(\pi)$ denotes the number of parts of $\pi$, and $\omega(\pi)$ the number of ones. The following Andrews-Beck type congruence, conjectured by Beck, was proved by Andrews \cite{Andrews-2021}:
$$
\sum_{m=1}^4mNT(m,5,5n+1)\equiv\sum_{m=1}^4mNT(m,5,5n+4)\equiv 0\pmod 5.
$$
Motivated by Andrews' works, variants of Andrews-Beck type congruence were studied by many authors. For example, Chern \cite{Ch-22} found and proved a list of this type congruences modulo 5,7,11, and 13. Moreover, identities and equalities involving these new partition statistics have been established by many authors through $q$-series approaches. For example, Mao \cite{Mao-22} obtained identities involving $NT(s,p,pn+d)$ such as
\begin{align}
\label{ex1}
\sum_{n=0}^\infty &\big(NT(1,7,7n+5)-NT(6,7,7n+5)\\
\nonumber
&+3NT(2,7,7n+5)-3NT(5,7,7n+5\big)q^n=-7\frac{(q^7;q^7)_\infty(q^3,q^4;q^7)_\infty}{(q,q^6;q^7)_\infty(q^2,q^5;q^7)_\infty^2}.
\end{align}
which is analogue of the well-known identity
$$
\sum_{n=1}^\infty p(5n+4)q^n=5\frac{(q^5;q^5)^5}{(q;q)^6},
$$
and also indicate the congruence
\begin{align*}
NT(1,7,7n+5)-NT(6,7,7n+5)+3NT(2,7,7n+5)-3NT(5,7,7n+5)\equiv 0 \pmod{7}.
\end{align*}
Here and throughout the paper
$$
(a;q)_n:=\prod_{k=0}^{n-1}(1-aq^k),
$$
$$
(a;q)_\infty:=\prod_{k=0}^{\infty}(1-aq^k),
$$
$$
(a_1,a_2,\cdots,a_n;q)_\infty:=\prod_{k=1}^{n}(a_k;q)_\infty,
$$
and for convenience
$$
(q)_\infty:=(q;q)_\infty.
$$
Identities involving $NT$ and $M_\omega$ are also established by many authors. For example, Jin, Liu and Xia \cite{JLX-22} found the relations such as
\beq
\label{ex2}
NT(2,5,5n+1)-NT(3,5,5n+1)=M_\omega(2,5,5n+1)-M_\omega(3,5,5n+1).
\eeq
In a recent work with Chen and Yin, the first author \cite{ccy} found the identity for $p=11$,
\beq
\label{ex3}
\sum_{m=1}^{5}m\left[M_{\omega}(m,11,11n+6)-M_{\omega}(11-m,11,11n+6)\right]=0.
\eeq

As stated in \cite{Andrews-2021}, the fact $N(s,k,n)=N(k-s,k,n)$ is generally false if $N$ is replaced by $NT$. Recently, Mao \cite{Mao-24} established the modular approach for analyzing the difference $NT(s,k,n)-NT(k-s,k,n)$ and proved additional identities of the form shown in \eqref{ex1}-\eqref{ex3} by applying the theory of mock modular forms. Mao \cite[Lemma 2.1]{Mao-24} represented the difference $NT(s,k,n)-NT(k-s,k,n)$ as the rank function and the (generalized) Appell-Lerch sum.

\begin{align}
\label{maoid}
&\sum_{n=0}^\infty\left(NT(s,k,n)-NT(k-s,k,n)\right)q^n\\
\nonumber
=&\sum_{j=1}^{k-1}\frac{\zeta_{k}^{j(s-1)}(1-\zeta_{k}^j)}{k}\left\{\frac{1}{2}+\frac{\zeta_{k}^{j/2}}{2\pi i (q)_\infty}\cdot \frac{\partial}{\partial u}\bigg|_{u=0} A_3\left(u-\frac{j}{k};\tau\right)\right\}\\
\nonumber
&-\sum_{j=1}^{k-1}\frac{\zeta_{k}^{js}(1+\zeta_{k}^j)\mathcal{R}\left(\zeta_{k}^j;q\right)}{2k(1-\zeta_{k}^j)}.
\end{align}

By analogy, the difference $M_\omega(m,k,n)-M_\omega(k-m,k,n)$ can be expressed as \cite[Lemma 2.2]{Mao-24}

\begin{align}
\label{maoidm}
&\sum_{n=0}^\infty\left(M_\omega(s,k,n)-M_\omega(k-s,k,n)\right)q^n\\
\nonumber
=&\sum_{j=1}^{k-1}\frac{\zeta_{k}^{j(s-1/2)}(1-\zeta_{k}^j)}{2k\pi i(q)_\infty}\cdot \frac{\partial}{\partial u}\bigg|_{u=0} A_1\left(u-\frac{j}{k};\tau\right)\\
\nonumber
&-\sum_{j=1}^{k-1}\frac{\zeta_{k}^{js}(1+\zeta_{k}^j)\mathcal{C}\left(\zeta_{k}^j;q\right)}{2k(1-\zeta_{k}^j)}.
\end{align}

The function $\mathcal{R}(z;q)$ (resp. $\mathcal{C}(z;q)$) is the rank (resp. crank) function,
$$
\mathcal{R}(z;q):=\sum_{n=0}^\infty\sum_{m=-\infty}^\infty N(m,n)z^mq^n,
$$
$$
\mathcal{C}(z;q):=\sum_{n=0}^\infty\sum_{m=-\infty}^\infty M(m,n)z^mq^n,
$$
and $A_\ell(u,v;\tau)$ is the (generalized) Appell-Lerch series
\beq
\label{eq:AppellLerch}
A_\ell(u,v;\tau):=a^{\ell/2}\sum_{n\in \mathbb{Z}}\frac{(-1)^{\ell n}q^{\ell n(n+1)/2}b^n}{1-aq^n},
\eeq
where $a:=e^{2\pi i u}$, $b:=e^{2\pi i v}$ and $q:=e^{2\pi i \tau}$. We also write $A_\ell(u;\tau)$:=$A_\ell(u,0;\tau)$.

In this paper, we start from \eqref{maoid} to find an analog of \eqref{hm}. Let $p\geq 5$ be prime and
$$
D_{NT}(s,p):=\sum_{n=0}^\infty\left(NT(s,p,n)-NT(p-s,p,n)\right)q^n-\sum_{r=1}^{p-1}\frac{p-2r}{2p}D(r-s,p),
$$
where we point out that
$$
\sum_{r=1}^{p-1}\frac{p-2r}{2p}D(r-s,p)=\sum_{r=1}^{p-1}\frac{p-2r}{2p}\sum_{n=0}^\infty N(r-s,p,n)q^n.
$$
Our main theorem (see Theorem \ref{mainN} below) shows that
$$
D_{NT}(s,p)=d_{NT}(s,p)+t_{s,p},
$$
where $d_{NT}(s,p)$ represents some given Appell-Lerch series and their derivatives, and $t_{s,p}$ is ``modular'' (weight $3/2$). For example, we have
$$
D_{NT}(1,5)=q^8L_5(2;q^5)-q^5L_5(4;q^5)+t_{1,5},
$$
where for prime $p$ and integer $0<v<p$, we define
\begin{align}
\label{lpp}
L_p(v;q):=&L_p(v)\\
\nonumber
:=&l_p(v)+\frac{(-1)^v}{(q^p;q^p)_\infty}\left(p\sum_{n=-\infty}^\infty\frac{(-1)^nq^{3pn(n+1)/2}}{(1-q^{pn+v})^2}
+\left(\frac{p}{2}-3v\right)\sum_{n=-\infty}^\infty\frac{(-1)^nq^{3pn(n+1)/2}}{1-q^{pn+v}}\right),
\end{align}
with $l_p(v)=(-1)^v\frac{p-6v}{2}q^{-v}$ if $0<v<\frac{p}{6}$, $l_p(v)=(-1)^v\frac{5p-6v}{2}q^{v-p}$ if $\frac{5p}{6}<v<p$ and $l_p(v)=0$ else. For $v=0$ we define

\begin{align}
\label{lp0}
L_p(0)=\frac{p}{(q^p;q^p)_\infty}\left(\sum_{n\neq 0} \frac{(-1)^nq^{3pn(n+1)/2}}{(1-q^{pn})^2}
+3\sum_{n=1}^\infty \frac{nq^{pn}}{1-q^{pn}}-\frac{1}{12}\right).
\end{align}
We remark that the function $L_p(0)$ is close to the Andrews $spt$ function
$$
\sum_{n=0}^\infty spt(n)q^n=\frac{1}{(q;q)_\infty}\left(\sum_{n\neq 0} \frac{(-1)^nq^{3n(n+1)/2}}{(1-q^n)^2}
+\sum_{n=1}^\infty \frac{nq^{pn}}{1-q^{pn}}\right),
$$
which is related to a weight 3/2 harmonic Maass form studied by Bringmann \cite{Bri08}, also Bringmann, Folsom and Ono \cite{BFO09}.

Furthermore, we find that all $p$-dissections of the part $t_{s,p}$ are modular functions on $\Gamma_1(p)$ when multiplied by certain eta-quotients and give the expression of $t_{s,p}$ (under $p$-dissection) by the modular function approach for $p=5,7$ (see Appendices \ref{apx:A} and \ref{apx:B}).

Then we show that Theorems 3.6 and 3.7 imply identities such as \eqref{ex1}--\eqref{ex3}.

\subsection{Appell-Lerch series of level $\ell$}
\label{subsec:Appell-Lerch series}
In proving the main theorems (Theorems \ref{mainN} and \ref{mainM} below) and finding out the exact forms of $N_p(s,k)$ and $M_p(s,k)$ in these theorems, we rely heavily on the properties of the modular completions $\widehat{A}_\ell(u;\tau)$ of generalized Appell-Lerch series $A_\ell(u;\tau)$. The completion was obtained by Zwegers \cite{Zw-19} which extends a previous important result obtained in the same author's Ph.D. Thesis. Zwegers also gave the transformation equations of these functions as real analytic Jacobi forms. In particular, $\widehat{A}_\ell(u;\tau)$ is a real analytic Jacobi form of weight $1$ and index $-\ell/2$; c.f. \cite[Remark 4]{Zw-19}. The positive integer $\ell$ is called the level.

The functions $A_\ell(u;\tau)$ appear in several branches of mathematics, e.g., in representation theory of affine Lie superalgebras, conformal field theory, the theory of K3 surfaces, the study of vector bundles on elliptic curves; c.f. \cite{KW-01}, \cite{STT-05}, \cite{MO-09}, \cite{Pol-01}. More related to our topic, generalized Appell-Lerch series play an important role in the study of Ramanujan's mock theta functions; c.f. \cite{BO-10} and \cite{Zag-09}.

Mao \cite{Mao-24} used $A_1(u;\tau)$ and $A_3(u;\tau)$ to derive some of his theorems (see \eqref{maoid} and \eqref{maoidm} above). In detail, Mao obtained the transformation equations of $\frac{\partial}{\partial u}|_{u=0} \widehat{A}_\ell\left(u-\frac{j}{k};\tau\right)$ and then proved the holomorphicity of the ``holomorphic parts'' of these functions at cusps for the level $\ell=1$ and $3$. These properties are key to the proof of his results.

In this paper, we obtain some new results on $\widehat{A}_\ell(u;\tau)$ which in the cases $\ell=1$ and $3$ will be used in the proof of Theorems \ref{mainN} and \ref{mainM} and in determining how many terms one should check when using a computer algebra system to prove identities on $N_p(s,k)$ and $M_p(s,k)$. These results are more  meticulous and comprehensive than Mao's and may be of independent interest due to the role of $\widehat{A}_\ell(u;\tau)$ in mathematics. We list the results as follows.
\begin{itemize}
    \item We obtain the $q$-series expansions of both the holomorphic part and nonholomorphic part of $\frac{\partial}{\partial u}\vert_{u=0}\widehat{A}_\ell(u-x;\tau)$ at any cusp, where $x$ is a positive nonintegral rational number. See Theorem \ref{thm:partAlhatCusp} below.
    \item We give an explicit formula for the order of the holomorphic part of $\frac{\partial}{\partial u}\vert_{u=0}\widehat{A}_\ell(u-x;\tau)$ at any cusp. See Corollary \ref{coro:orderPartialAlhat} below.
    \item We establish the modularity of certain linear combination of $\widehat{A}_\ell$ and its derivative (with respect to $u$) and give an explicit formula for the orders of this function at cusps. See Proposition \ref{prop:glxModular} and Corollary \ref{cor:glxOrder} below.
    \item We show that the image of $\frac{\partial}{\partial u}\vert_{u=0}\widehat{A}_\ell(u-x;\tau)$ under an operator $U_{p,k}$, multiplied by some generalized Dedekind eta functions, is a real analytic modular form on the congruence subgroup $\Gamma_1(p)$ with trivial multiplier system. See Theorem \ref{thm:AppellGEtaModular} below. This exhibits a technique of how one can multiply a (possibly nonholomorphic) modular form by some generalized Dedekind eta functions to reduce the multiplier system to a trivial one.
\end{itemize}

In the proofs of the above results, we freely use the theory of Jacobi forms for which the reader may consult \cite{EZ-85}. Finally, we emphasize that the real analytic modular forms mentioned above, in general, are not harmonic Maass forms as the shapes of their nonholomorphic parts show.

\subsection{Organization of the paper and notations}

This paper is organized as follows. In Section \ref{sec:Modularity}, we derive some properties of the generalized Appell-Lerch series $A_\ell(u;\tau)$ as explained in Section \ref{subsec:Appell-Lerch series}. In Section \ref{sec:Generating functions}, we find out the functions to cancel the nonholomophic part of each $p$-dissection of the $NT$ function and prove the main theorem. In Section \ref{sec:Examples and the algorithm}, we calculate the lower bounds of the orders at all cusps for $N_p(s,k)$ and $M_p(s,k)$ in Theorems~\ref{mainN} and~\ref{mainM} and then use the valence formula to express $N_p(s,k)$ and $M_p(s,k)$ as linear combinations of certain generalized Dedekind eta quotients. The resulted expressions are collected in Appendices \ref{apx:A} and \ref{apx:B}.

We list some common notations here. The symbol $\uhp$ denotes the upper half plane $\{z\in\C\colon\IM z>0\}$ and $\tau$ is tacitly assumed to take values in $\uhp$. The symbol $q$ always means $e^{2\pi i\tau}$. If $a, b$ are two numbers, the notation $\delta_{a,b}$ refers to the Kronecker $\delta$, that is, $\delta_{a,b}=1$ if and only if $a=b$. Let $x$ be a rational number; we define $\delta_x=1$ if $x\in\Z$ and $\delta_x=0$ otherwise; define $\sgn(x)=1$ or $-1$ according to $x>0$ or $x<0$ respectively and set $\sgn(0)=0$. For integers $a$ and $b$, $(a,b)$ means the greatest common divisor. The notations $\zeta_k=e^{\frac{2\pi i}{k}}$ and $s_k=\frac{k^2-1}{24}$ have been mentioned in above. Let $p\geq5$ be a prime; we set $\chi_{12}(p)=1$ if $p\equiv1,11\pmod{12}$ and $\chi_{12}(p)=-1$ if $p\equiv5,7\pmod{12}$. For a complex-valued real-differentiable function $f(u)$ of a complex variable $u$, define $\frac{\partial}{\partial u}f=\frac{1}{2}\left(\frac{\partial f}{\partial \RE u}-\rmi\frac{\partial f}{\partial \IM u}\right)$. If $f$ is holomorphic, this coincides with the complex derivative.

The slash operators $f\vert_k\gamma$ acting on modular forms or analogues $f$ are defined preceding Theorem \ref{thm:partAlhatCusp} and after Corollary \ref{cor:glxOrder}. The slash operators $\phi\vert_k\gamma$ acting on Jacobi forms or analogues $\phi$ are used in some proofs and their definition can be found in \cite[Theorem 1.4]{EZ-85}. The order of a modular form or its analogue at a cusp, $\ord\nolimits_{\rmi\infty}f\vert_r\gamma$ or $\ord_{a/c}f$, is defined preceding the statement of and after the proof of Corollary \ref{coro:orderPartialAlhat}. The divisor $\mathrm{div}_{a/c}f$ or $\mathrm{div}_{\tau}f$ of a meromorphic modular form $f$ at $a/c$ or $\tau$, each of which is a term in the left-hand side of the valence formula, is defined in the beginning part of Section \ref{subsec:Behaviour at cusps}.

\section{Modularity}
\label{sec:Modularity}

\subsection{Appell-Lerch series}

Following Zwegers \cite{Zw-19}, let
\begin{equation}
\label{eq:defAlhat}
\widehat{A}_\ell(u;\tau):=A_\ell (u;\tau)+\frac{i}{2}\sum_{m=0}^{\ell-1}e^{2\pi imu}\theta\left(m\tau+\frac{\ell-1}{2};\ell\tau\right) R\left(\ell u-m\tau-\frac{\ell-1}{2};\ell \tau\right),
\end{equation}
where $\tau\in\uhp$, $u\in\C\setminus(\Z\tau\oplus\Z)$,
\begin{align}
\theta(z;\tau):=&\sum_{\nu \in \frac{1}{2}+\mathbb{Z}}e^{\pi i \nu^2 \tau+2\pi i \nu (z+\frac{1}{2})}\notag\\
=&-iq^{\frac{1}{8}}e^{-\pi i z}\prod_{n=1}^{\infty}(1-q^n)(1-e^{2\pi iz}q^{n-1})(1-e^{-2\pi i z}q^n),\label{eq:triProd}
\end{align}
and
\begin{align*}
R(u;\tau):=\sum_{\nu \in \frac{1}{2}+\mathbb{Z}}&\left\{\sgn(v)-E\left((\nu+\frac{\mathrm{Im}(u)}{\mathrm{Im}(\tau)})\sqrt{2\mathrm{Im}(\tau)}\right)\right\}\\
&\times (-1)^{\nu -\frac{1}{2}}q^{\frac{-\nu^2}{2}}e^{-2\pi i \nu u},
\end{align*}
with
$$
E(z):=2\int_{0}^z e^{-\pi u^2} \mathrm{d}u=\sgn(z)(1-\beta(z^2)) \quad (z\in \mathbb{R}),
$$
$$
\beta(x):=\int_x^\infty u^{-\frac{1}{2}}e^{-\pi u} \mathrm{d}u \quad (x\in \mathbb{R}_{\geq 0}).
$$
One can easily verify that
$$
\widehat{A}_1(u;\tau)=A_1(u;\tau),
$$
and
\beq
\label{a3m}
\widehat{A}_3(u;\tau)=A_3(u;\tau)+\frac{i}{2}\sum_{m=1}^2e^{2\pi i m u}\theta(m\tau;3\tau)R(3u-m\tau;3\tau).
\eeq

Zwegers \cite[Theorem 4]{Zw-19} and Mao \cite[Theorem 2.4]{Mao-24} found the transformation formulas for these $\widehat{A}_\ell$ and their derivatives which we recall here. For $m,n\in \mathbb{Z}$,
$$
\widehat{A}_\ell(u+n\tau+m;\tau)=(-1)^{\ell (n+m)}e^{2\pi i\ell nu} q^{\ell n^2/2}\widehat{A}_\ell(u;\tau).
$$
For $\abcdMAT \in \mathrm{SL}_2(\mathbb{Z})$,
\beq
\label{Alt}
\widehat{A}_\ell\left(\frac{u}{c\tau+d};\frac{a\tau+b}{c\tau+d}\right)=(c\tau+d)e^{\frac{-\pi i c \ell u^2}{c\tau+d}}\widehat{A}_\ell(u;\tau).
\eeq
For $\abcdMAT \in \mathrm{SL}_2(\mathbb{Z})$ and $j\not\equiv 0 \pmod{k}$,
\begin{align*}
\frac{\partial}{\partial u}\bigg|_{u=0} \widehat{A}_\ell\left(u-\frac{j}{k};\frac{a\tau+b}{c\tau+d}\right)
=&\frac{2j\ell c \pi i(c\tau+d)^2}{k}e^{\frac{-j^2\ell c \pi i(c\tau+d)}{k^2}} \widehat{A}_\ell\left(-\frac{j(c\tau+d)}{k};\tau\right)\\
&+(c\tau+d)^2 e^{\frac{-j^2\ell c \pi i(c\tau+d)}{k^2}} \frac{\partial}{\partial u}\bigg|_{u=0} \widehat{A}_\ell\left(u-\frac{j(c\tau+d)}{k};\tau\right).
\end{align*}
For $\abcdMAT \in \Gamma_1(k)$, we have\footnote{In \cite[Eq. (2.18)]{Mao-24}, the factor $e^{\frac{j^2\ell c \pi i(2-d)}{k^2}}$ occurred instead of $e^{\frac{j^2\ell cd \pi i}{k^2}}$. The two factors are equal since $\tbtmat{a}{b}{c}{d}\in\Gamma_1(k)$. However, we prefer the one presented here since the transformation formulas can be generalized to $\partial\widehat{A}_\ell(-j/k)\vert_2\tbtmat{a}{b}{c}{d}=(-1)^{\ell cj/k}e^{\frac{j^2\ell cd \pi i}{k^2}}\partial\widehat{A}_\ell(-dj/k),\,\tbtmat{a}{b}{c}{d}\in\Gamma_0(k)$ where we can not use the other factor. For the notation see Section \ref{sec:Behaviour at cusps}.}
\begin{equation}
\label{eq:transformationGamma1k}
\frac{\partial}{\partial u}\bigg|_{u=0} \widehat{A}_\ell\left(u-\frac{j}{k};\frac{a\tau+b}{c\tau+d}\right)
=(-1)^{\frac{\ell j(c+d-1)}{k}}(c\tau+d)^2 e^{\frac{\pi i\ell cdj^2 }{k^2}} \frac{\partial}{\partial u}\bigg|_{u=0} \widehat{A}_\ell\left(u-\frac{j}{k};\tau\right).
\end{equation}

\subsection{Behavior of $\widehat{A}_\ell$ and $\partial\widehat{A}_\ell$ at cusps}
\label{sec:Behaviour at cusps}

Let $x$ be a positive rational number that is not an integer (say $x=j/k$), and $\gamma=\tbtmat{a}{b}{c}{d}\in\slZ$ with $c\geq0$. Let $\partial\widehat{A}_\ell$ denote the function $\frac{\partial}{\partial u}\widehat{A}_\ell(u;\tau)$ where $\frac{\partial}{\partial u}=\frac{1}{2}\left(\frac{\partial}{\partial \RE u}-\rmi\frac{\partial}{\partial \IM u}\right)$. Therefore $\partial\widehat{A}_\ell(-x)=\partial\widehat{A}_\ell(-x;\tau):=\left.\frac{\partial}{\partial u}\right\vert_{u=0}\widehat{A}_\ell(u-x;\tau)$. The weight $2$ slash operator is defined by $\partial\widehat{A}_\ell(-x)\vert_2\gamma(\tau):=(c\tau+d)^{-2}\partial\widehat{A}_\ell\left(-x;\frac{a\tau+b}{c\tau+d}\right)$.
\begin{theorem}
\label{thm:partAlhatCusp}
The function $\frac{1}{2\uppi\rmi}\rme^{-\uppi\rmi \ell cdx^2}\cdot\partial\widehat{A}_\ell(-x)\vert_2\gamma$ can be written as $H_1+H_2$. In this decomposition, $H_1(\tau)$, the holomorphic part, is given by
\beqs
C+\left(\sum_{n>-cx,\,m\geq M}-\sum_{n<-cx,\,m< M}\right)(-1)^{\ell n}\rme^{2\uppi\rmi(m-M)dx}\cdot(m-M)q^{\frac{\ell}{2}(n+cx)^2+(n+cx)(m-M)},
\eeqs
where the variables $n$ and $m$ take values in $\mathbb{Z}$, $M=\ell cx+\frac{\ell}{2}$ and
\beqs
C=-\delta_{cx}\frac{1-\delta_{(\ell+1)/2}+\delta_{(\ell+1)/2}\cdot(-1)^{cx}\cos(\uppi dx)}{4\sin^2(\uppi dx)}.
\eeqs
On the other hand, the nonholomorphic part $H_2(\tau)$ is given by
\beqs
-\frac{1}{2\uppi}\sum_{n,m\in\Z}(-1)^{\ell n}e^{2\pi i dx(\ell(n-cx)-(m+\frac{\ell}{2}))}f\left(\ell(n-cx)-(m+\frac{\ell}{2});\IM\tau\right)q^{-\frac{\ell}{2}(n-cx)^2+(n-cx)(m+\frac{\ell}{2})}
\eeqs
where
\beqs
f(t;\IM\tau):=\uppi \abs{t}\beta\left(t^2\cdot\frac{2\IM\tau}{\ell}\right)-\rme^{-\uppi t^2\cdot\frac{2\IM\tau}{\ell}}\cdot\sqrt{\frac{\ell}{2\IM\tau}}.
\eeqs
\end{theorem}
\begin{proof}
We shall use the theory of Jacobi forms in this proof, c.f. \cite{EZ-85}. According to Zwegers \cite[Theorem 4]{Zw-19}, $\widehat{A}_\ell(u;\tau)$ transforms like a Jacobi form of weight $1$ and index $-\ell/2$. One can verify by \cite[Theorem 1.4]{EZ-85} that, for $\gamma_1=\left[\tbtmat{a_1}{b_1}{c_1}{d_1},(\lambda,\mu),\xi\right]$,
\begin{equation}
\label{eq:derivativeJacobi}
(\partial\widehat{A}_\ell)\vert_{2,-\ell/2}\gamma_1(u;\tau)=\partial(\widehat{A}_\ell\vert_{1,-\ell/2}\gamma_1)(u;\tau)-2\pi i\left(\frac{c_1}{c_1\tau+d_1}\ell(u+\tau\lambda+\mu)-\ell\lambda\right)\widehat{A}_\ell\vert_{1,-\ell/2}\gamma_1(u;\tau),
\end{equation}
where $a_1d_1-b_1c_1=1$ and $\lambda,\mu\in\R$, $\abs{\xi}=1$. It follows that for the given $\gamma=\tbtmat{a}{b}{c}{d}$ (even when $c<0$), we have
\begin{align}
&\partial\widehat{A}_\ell\vert_{2,-\ell/2}\left[(0,-x),1\right]\vert_{2,-\ell/2}\tbtmat{a}{b}{c}{d}(u;\tau)\notag\\
=&\partial\widehat{A}_\ell\vert_{2,-\ell/2}\left[\tbtmat{a}{b}{c}{d},(-xc,-xd),1\right](u;\tau)\notag\\
=&\partial(\widehat{A}_\ell\vert_{1,-\ell/2}[(-cx,-dx),1])(u;\tau)-2\pi i\frac{c}{c\tau+d}\ell u\widehat{A}_\ell\vert_{1,-\ell/2}[(-cx,-dx),1](u;\tau)\label{eq:pAlvertxabcd},
\end{align}
in the last equality of which we have used the fact $\widehat{A}_\ell\vert_{1,-\ell/2}\tbtmat{a}{b}{c}{d}=\widehat{A}_\ell$ (c.f. \cite[Theorem 4]{Zw-19}). Unfolding the slash operators, setting $u=0$ and using the chain rule for $\frac{\partial}{\partial u}$ in the above equality we obtain
\begin{align}
&\partial\widehat{A}_\ell(-x)\vert_2\gamma=q^{-\frac{\ell c^2x^2}{2}}e^{-\pi i\ell cdx^2}\cdot\left(2\pi i\ell cx\widehat{A}_\ell(-x(c\tau+d);\tau)+\partial\widehat{A}_\ell(-x(c\tau+d);\tau)\right)\notag\\
=&q^{-\frac{\ell c^2x^2}{2}}e^{-\pi i\ell cdx^2}\cdot\left(2\pi i\ell cx(A_\ell(-x(c\tau+d);\tau)+\widetilde{A}_\ell(-x(c\tau+d);\tau))\right.\label{eq:pAlhatDecomp}\\
&+\left.\partial A_\ell(-x(c\tau+d);\tau)+\partial\widetilde{A}_\ell(-x(c\tau+d);\tau)\right),\notag
\end{align}
where $\widetilde{A}_\ell=\widehat{A}_\ell-A_\ell$ is the nonholomorphic correction term (c.f. \eqref{eq:defAlhat}). It remains to compute the four terms
\begin{equation*}
A_\ell(-x(c\tau+d);\tau),\quad \partial A_\ell(-x(c\tau+d);\tau),\quad \widetilde{A}_\ell(-x(c\tau+d);\tau),\quad \partial\widetilde{A}_\ell(-x(c\tau+d);\tau).
\end{equation*}

By the definition,
\begin{equation*}
A_\ell(-x(c\tau+d);\tau)=e^{-\pi i\ell dx}q^{-\frac{\ell cx}{2}}\sum_{n\in\Z}\frac{(-1)^{\ell n}q^{\ell n(n+1)/2}}{1-e^{-2\pi idx}q^{n-cx}}.
\end{equation*}
Splitting the sum $\sum_{n\in\Z}$ into three sums $\sum_{n>cx}$, $\sum_{n<cx}$ and $\sum_{n=cx}$ and expanding the geometric series $(1-z)^{-1}=1+z+z^2+\dots$ we find that
\begin{multline}
\label{eq:AlxExpan}
A_\ell(-x(c\tau+d);\tau)=e^{-\pi i\ell dx}q^{-\frac{\ell cx}{2}}\cdot\left(\sum_{n>cx,\, m\geq0}(-1)^{\ell n}q^{\ell n(n+1)/2}e^{-2\pi imdx}q^{m(n-cx)}\right.\\
+\sum_{n<cx,\, m\geq0}(-1)^{\ell n+1}q^{\ell n(n+1)/2}e^{-2\pi idx}q^{cx-n}e^{2\pi imdx}q^{m(cx-n)}\\
\left.+\delta_{cx}(-1)^{\ell cx}q^{\ell cx(cx+1)/2}(1-e^{-2\pi idx})^{-1}\right).
\end{multline}
The two series on the right-hand side converge normally\footnote{All the series in this proof converge normally (on compact sets of $\uhp$) and we will never mention this hereafter.} so the equality actually holds for $\tau\in\uhp$.

For the term $\partial A_\ell(-x(c\tau+d);\tau)$, we have
\begin{equation}
\label{eq:pAlxExpan}
\partial A_\ell(-x(c\tau+d);\tau)=\pi i\ell A_\ell(-x(c\tau+d);\tau)+2\pi ie^{-2\pi ix(c\tau+d)(\ell/2+1)}\sum_{n\in\Z}\frac{(-1)^{\ell n}q^{(\ell n^2+\ell n + 2n)/2}}{(1-e^{-2\pi idx}q^{n-cx})^2}.
\end{equation}
The first term on the right-side hand has been computed in \eqref{eq:AlxExpan} and second term, the sum $\sum_{n\in\Z}$ is equal to
\begin{multline}
\label{eq:pAlxExpanSecond}
\sum_{n>cx,\,m\geq0}(-1)^{\ell n}q^{n(\ell n+\ell+2)/2}e^{-2\pi imdx}(m+1)q^{m(n-cx)}\\
+\sum_{n<cx,\,m\geq0}(-1)^{\ell n}q^{n(\ell n+\ell+2)/2}e^{4\pi idx}q^{2(cx-n)}e^{2\pi imdx}(m+1)q^{m(cx-n)}\\
+\delta_{cx}(-1)^{\ell cx}q^{cx(\ell cx+\ell+2)/2}(1-e^{-2\pi idx})^{-2}.
\end{multline}
Combining \eqref{eq:AlxExpan}, \eqref{eq:pAlxExpan} and \eqref{eq:pAlxExpanSecond} we obtain
\begin{equation}
\label{eq:AlpAldecomp}
q^{-\frac{\ell c^2x^2}{2}}e^{-2\pi i\ell cdx^2}\left(\ell cx(A_\ell(-x(c\tau+d);\tau)+\frac{1}{2\pi i}\partial A_\ell(-x(c\tau+d);\tau)\right)=-S_1+S_2+C_1,
\end{equation}
where
\begin{align*}
S_1&=\sum_{n<-cx,\,m\leq0}(-1)^{\ell n}\rme^{2\uppi\rmi(m-M)dx}\cdot(m-M)q^{\frac{\ell}{2}(n+cx)^2+(n+cx)(m-M)},\\
S_2&=\sum_{n>-cx,\,m\geq1}(-1)^{\ell n}\rme^{2\uppi\rmi(m-M)dx}\cdot(m-M)q^{\frac{\ell}{2}(n+cx)^2+(n+cx)(m-M)},\\
C_1&=\delta_{cx}(-1)^{\ell cx}e^{-2\pi i\ell cdx^2}\left(\ell\left(cx+\frac{1}{2}\right)\frac{e^{-\pi i\ell dx}}{1-e^{-2\pi idx}}+\frac{e^{-\pi i(\ell+2)dx}}{(1-e^{-2\pi idx})^2}\right).
\end{align*}

Now we consider the terms $\widetilde{A}_\ell(-x(c\tau+d);\tau)$ and $\partial\widetilde{A}_\ell(-x(c\tau+d);\tau)$. As a prerequisite, note that for $\lambda,\mu\in\Q$ we have
\begin{multline}
\label{eq:Rtau}
R(\lambda\tau+\mu;\tau)=\sum_{\nu\in\frac{1}{2}+\Z}(\sgn(\nu)-\sgn(\nu+\lambda))(-1)^{\nu-\frac{1}{2}}e^{-2\pi i\nu\mu}q^{-\frac{1}{2}\nu^2-\lambda \nu}\\
+\sum_{\nu\in\frac{1}{2}+\Z}\sgn(\nu+\lambda)\beta((\nu+\lambda)^2\cdot2\IM\tau)(-1)^{\nu-\frac{1}{2}}e^{-2\pi i\nu\mu}q^{-\frac{1}{2}\nu^2-\lambda \nu}
\end{multline}
where the first sum on the right-hand side is the holomorphic part and the second sum the nonholomorphic part. On the other hand, since
\begin{align*}
\frac{\partial}{\partial u}E\left((\nu+\frac{\mathrm{Im}(u)}{\mathrm{Im}(\tau)})\sqrt{2\mathrm{Im}(\tau)}\right)&=-\frac{i}{2}\frac{\partial}{\partial\IM u}E\left((\nu+\frac{\mathrm{Im}(u)}{\mathrm{Im}(\tau)})\sqrt{2\mathrm{Im}(\tau)}\right)\\
&=-ie^{-\pi(\nu+\frac{\IM z}{\IM\tau})^2\cdot2\IM\tau}\cdot\sqrt{\frac{2}{\IM\tau}},
\end{align*}
we have
\begin{multline}
\label{eq:pRtau}
\partial R(\lambda\tau+\mu;\tau)=\sum_{\nu\in\frac{1}{2}+\Z}(\sgn(\nu)-\sgn(\nu+\lambda))(-1)^{\nu-\frac{1}{2}}(-2\pi i\nu)e^{-2\pi i\nu\mu}q^{-\frac{1}{2}\nu^2-\lambda \nu}\\
+\sum_{\nu\in\frac{1}{2}+\Z}\sgn(\nu+\lambda)\beta((\nu+\lambda)^2\cdot2\IM\tau)(-1)^{\nu-\frac{1}{2}}(-2\pi i\nu)e^{-2\pi i\nu\mu}q^{-\frac{1}{2}\nu^2-\lambda \nu}\\
+i\sum_{\nu\in\frac{1}{2}+\Z}e^{-\pi(\nu+\lambda)^2\cdot2\IM\tau}\cdot\sqrt{\frac{2}{\IM\tau}}(-1)^{\nu-\frac{1}{2}}e^{-2\pi i\nu\mu}q^{-\frac{1}{2}\nu^2-\lambda \nu}
\end{multline}
where the first sum on the right-hand side is the holomorphic part and the remaining two sums the nonholomorphic part. By definition
\begin{multline*}
\ell cx(\widetilde{A}_\ell(-x(c\tau+d);\tau)+\frac{1}{2\pi i}\partial \widetilde{A}_\ell(-x(c\tau+d);\tau)\\
=\frac{1}{2\pi i}\sum_{k=0}^{\ell-1}q^{-kcx}e^{-2\pi ikdx}\theta\left(k\tau+\frac{\ell-1}{2};\ell\tau\right)\cdot\left(-\pi(\ell cx+k)R\left(-(cx+\frac{k}{\ell})\ell\tau-\ell dx-\frac{\ell-1}{2};\ell\tau\right)\right.\\
+\left.\frac{i\ell}{2}\partial R\left(-(cx+\frac{k}{\ell})\ell\tau-\ell dx-\frac{\ell-1}{2};\ell\tau\right)\right).
\end{multline*}
It follows from this, \eqref{eq:Rtau}, \eqref{eq:pRtau} and the definition \eqref{eq:triProd} of $\theta$ that the holomorphic part of $q^{-\frac{\ell c^2x^2}{2}}e^{-2\pi i\ell cdx^2}\left(\ell cx(\widetilde{A}_\ell(-x(c\tau+d);\tau)+\frac{1}{2\pi i}\partial \widetilde{A}_\ell(-x(c\tau+d);\tau)\right)$ is equal to
\begin{multline}
\label{eq:RholoPart}
\frac{1}{2\pi i}q^{-\frac{\ell c^2x^2}{2}}e^{-2\pi i\ell cdx^2}\sum_{0\leq k<\ell}\sum_{n\in\frac{1}{2}+\Z}\sum_{\nu\in\frac{1}{2}+\Z}e^{-2\pi ikdx}q^{-kcx}e^{\pi i\ell n}q^{\frac{\ell}{2}n^2+kn}\\
\cdot\left(\sgn(\nu)-\sgn(\nu-cx-\frac{k}{\ell})\right)(-1)^{\nu-\frac{1}{2}}\pi(\ell\nu-\ell cx-k)e^{2\pi i\nu(\ell dx+\frac{\ell-1}{2})}q^{-\frac{\ell}{2}\nu^2+(\ell cx+k)\nu}.
\end{multline}
Note that $\sgn(\nu)-\sgn(\nu-cx-\frac{k}{\ell})=0$ unless $0<\nu\leq cx+\frac{k}{\ell}$ since $cx+\frac{k}{\ell}\geq0$. Therefore, applying the changes of variables $n=\frac{1}{2}+n_1$ and $\nu=\frac{1}{2}+\nu_1$ we find that \eqref{eq:RholoPart} is equal to
\begin{multline}
\label{eq:RholoPart2}
(-1)^{\ell+1}e^{-2\pi i\ell cdx^2}\sum_{0\leq k<\ell}\sum_{n_1\in\Z}\sum_{\substack{{0\leq\nu_1\in\Z} \\ {\nu_1<cx+k/\ell-1/2}}}e^{\pi i\ell n_1}q^{\frac{\ell}{2}(n_1+\frac{1}{2}-cx)+(\ell cx+k)(n_1+\frac{1}{2}-cx)}\\
\cdot(-1)^{\nu_1}(\ell\nu_1+\frac{\ell}{2}-\ell cx-k)e^{2\pi i\left(\nu_1(\ell dx+\frac{\ell-1}{2})+(\frac{\ell}{2}-k)dx\right)}q^{-\frac{\ell}{2}(\nu+\frac{1}{2})^2+(\ell cx+k)(\nu+\frac{1}{2})}.
\end{multline}
Set $n=n_1-\nu_1-1$ and $m=\ell\nu+\ell-k-1$, which gives a bijection\footnote{It may happen that the domain and codomain of this bijection are the empty set.} from
$$
\{(k,n_1,\nu_1)\in\Z^3\colon 0\leq k<\ell,\, 0\leq\nu_1<cx+k/\ell-1/2\}
$$
onto
$$
\{(n,m)\in\Z^2\colon 0\leq m<M-1\}.
$$
Hence \eqref{eq:RholoPart2} becomes
\begin{equation*}
-S_3:=-\sum_{\substack{{n\in\Z} \\{1\leq m<M}}}(-1)^{\ell n}\rme^{2\uppi\rmi(m-M)dx}\cdot(m-M)q^{\frac{\ell}{2}(n+cx)^2+(n+cx)(m-M)}.
\end{equation*}
It follows from this, \eqref{eq:AlpAldecomp} and \eqref{eq:pAlhatDecomp} that the holomorphic part of $\frac{1}{2\uppi\rmi}\rme^{-\uppi\rmi \ell cdx^2}\cdot\partial\widehat{A}_\ell(-x)\vert_2\gamma$ is equal to $C_1-S_1+S_2-S_3$. Now if $cx\not\in\Z$, then $C=C_1=0$ and
\begin{equation*}
-S_1+S_2-S_3=\sum_{n>-cx,\,m\geq M}-\sum_{n<-cx,\,m< M}.
\end{equation*}
Thus, the holomorphic part of $\frac{1}{2\uppi\rmi}\rme^{-\uppi\rmi \ell cdx^2}\cdot\partial\widehat{A}_\ell(-x)\vert_2\gamma$ equals $H_1(\tau)$ as required. Otherwise if $cx\in\Z$, then the holomorphic part is equal to
\begin{equation*}
C_1-S_1+S_2-S_3=C_1-\sum_{n=-cx,\, 1\leq m<M} + \sum_{n>-cx,\,m\geq M}-\sum_{n<-cx,\,m< M}.
\end{equation*}
Since $C_1-\sum_{n=-cx,\, 1\leq m<M}=C$, the holomorphic part is equal to $H_1(\tau)$ as well in this case.

It remains to compute the nonholomorphic part of $\frac{1}{2\uppi\rmi}\rme^{-\uppi\rmi \ell cdx^2}\cdot\partial\widehat{A}_\ell(-x)\vert_2\gamma$ which is equal to the nonholomorphic part of $q^{-\frac{\ell c^2x^2}{2}}e^{-2\pi i\ell cdx^2}\left(\ell cx(\widetilde{A}_\ell(-x(c\tau+d);\tau)+\frac{1}{2\pi i}\partial \widetilde{A}_\ell(-x(c\tau+d);\tau)\right)$. This is equal to, by \eqref{eq:Rtau} and \eqref{eq:pRtau},
\begin{multline*}
\frac{1}{2\pi i}q^{-\frac{\ell c^2x^2}{2}}e^{-2\pi i\ell cdx^2}\sum_{0\leq k<\ell}q^{-kcx}e^{-2\pi ikdx}\theta\left(k\tau+\frac{\ell-1}{2};\ell\tau\right)\\
\cdot\sum_{\nu\in\frac{1}{2}+\Z}(-1)^{\nu-\frac{1}{2}}e^{2\pi i\nu(\ell dx+\frac{\ell-1}{2})}f\left(\ell\nu-\ell cx-k;\IM\tau\right)q^{-\frac{\ell}{2}\nu^2+(\ell cx+k)\nu}.
\end{multline*}
Substituting the definition of $\theta$ in the above expression and applying the change of variable $\nu=\nu'+\frac{1}{2}$ we find that the nonholomorphic part equals
\begin{multline*}
-\frac{1}{2\pi}e^{\pi i\ell cx}\sum_{\substack{{\nu',m'\in\Z} \\ {0\leq k<\ell}}}e^{2\pi i\left(dx(\ell\nu'+\ell-k-1-\frac{\ell}{2}-\ell cx+1)+\frac{1}{2}(\ell m'+k+\frac{\ell}{2}+\ell\nu'+\ell-k-1-\frac{\ell}{2}-\ell cx+1)\right)}\\
\cdot f\left(\ell\nu'+\frac{\ell}{2}-\ell cx-k;\IM\tau\right)q^{\frac{1}{2\ell}\left((\ell m'+k+\frac{\ell}{2})^2-(\ell\nu'+\ell-k-1-\frac{\ell}{2}-\ell cx+1)^2\right)}
\end{multline*}
Setting $\nu_1=\ell\nu'+\ell-k-1$ and $m=\ell m'+k$ we simplify the above expression further to
\begin{equation*}
-\frac{1}{2\pi}e^{\pi i\ell cx}\cdot\sum_{\substack{\nu_1,m\in\Z \\ \nu_1+m\equiv-1\bmod{\ell}}}e^{2\pi i\left((dx+\frac{1}{2})(\nu_1-M)+\frac{1}{2}(m+\frac{\ell}{2})\right)}f(\nu_1-M;\IM\tau)q^{\frac{1}{2\ell}\left((m+\frac{\ell}{2})^2-(\nu_1-M)^2\right)}.
\end{equation*}
Finally, the above expression is actually equal to $H_2(\tau)$ which can be seen by the change of variables $\nu_1=\ell n-m-1$. This concludes the whole proof.
\end{proof}

\begin{remark}
\label{rema:holoPart}
The terms ``holomorphic part'' and ``nonholomorphic part'' need to be further clarified since a decomposition of a real analytic function into a sum of a holomprhic one and a nonholomorphic one is not unique. Let $g$ be a real analytic function of period $N$, say $g=\partial\widehat{A}_\ell(-x)\vert_2\gamma$. Then $g(\tau)=\sum_{n\in N^{-1}\Z}c_n(\IM\tau)q^n$ which converges unconditionally for any fixed $\IM\tau$ with respect to the inner product on $\R/N\Z$ by elementary Fourier analysis. The uniquely defined coefficients $c_n(y)=N^{-1}\int_0^Ng(x+iy)e^{-2\uppi in(x+iy)}\,\mathrm{d}x$ are real analytic since $g$ is. Applying integration by parts to this integral we obtain $c_n(y)e^{-2\uppi iny}=O(n^{-2})$ for $y$ in any compact set, from which it follows that the series of $g(\tau)$ converges absolutely and uniformly on any compact subset of $\uhp$. If $\lim_{\IM\tau\to+\infty}c_n(\IM\tau)=c_n(+\infty)$ exists in $\C$ for each $n$ (this is the case for $g=\partial\widehat{A}_\ell(-x)\vert_2\gamma$), then formally $g$ can be decomposed as $H_1+H_2$ where
\begin{align*}
H_1(\tau)&=\sum_{n\in N^{-1}\Z}c_n(+\infty)q^n,\\
H_2(\tau)&=\sum_{n\in N^{-1}\Z}\left(c_n(\IM\tau)-c_n(+\infty)\right)q^n.
\end{align*}
However, the series $H_1(\tau)$ (or $H_2(\tau)$) not necessarily converges. To ensure this decomposition makes sense, we should \emph{assume} that $H_2(\tau)$ converges absolutely and uniformly on the sets $\{\tau\in\uhp\colon\IM\tau\geq Y_0\}$ with $Y_0$ being any positive real number from which the absolute and compact uniform convergence of $H_1$ follows. (This is the case for $g=\partial\widehat{A}_\ell(-x)\vert_2\gamma$ and as well for $g$ being any Harmonic Maass form.) Under these assumptions we call $H_1$, $H_2$ the holomorphic part and the nonholomorphic part of $g$ respectively. Note that the nonholomorphic part $H_2$ of $g$ is characterized by the conditions $g-H_2$ is holomorphic on $\uhp$ and the Fourier expansion $H_2(\tau)=\sum_{n\in N^{-1}\Z}c_n'(\IM\tau)q^n$ satisfies $\lim_{\IM\tau\to+\infty}c_n'(\IM\tau)=0$ for each $n$. Also note that our definition does not apply to all cases, e.g., to Harmonic Maass forms of manageable growth of weight $\leq1$.
\end{remark}
The following property is immediate according to Remark \ref{rema:holoPart}.
\begin{cor}
\label{coro:Hpart}
Let $f_1$ and $f_2$ have well defined holomorphic parts $H_{11}, H_{21}$ and nonholomorphic parts $H_{12}, H_{22}$ respectively. Then the holomorphic part of $f_1+f_2$ is $H_{11}+H_{21}$ and the nonholomorphic part of $f_1+f_2$ is $H_{12}+H_{22}$. 
\end{cor}

A function series $\sum_nf_n(z)$ is said to converge normally for $z\in A\subseteq\C$ if there is a sequence $c_n\geq0$ with the property $\abs{f_n(z)}\leq c_n,\,z\in A$ and $\sum_n c_n<+\infty$. Note that normal convergence implies absolute and uniform convergence and that the product of two normally convergent function series still converges normally. Also note that $H_1(\tau)$ and $H_2(\tau)$ in Theorem \ref{thm:partAlhatCusp} both converge normally on $\{\tau\in\uhp\colon\IM\tau\geq Y_0\}$ where $Y_0>0$.
\begin{prop}
\label{prop:productHolo}
Let $f_1,\,f_2$ be real analytic (or of class $C^2$ in general) functions on $\uhp$ of period $N$. ($N$ is a positive integer.) Suppose
\begin{enumerate}
    \item $f_1$ has well defined holomorphic part $H_{1}$ and nonholomorphic part $H_{2}$,
    \item The Fourier expansion of $H_2(\tau)$ converges normally on $\{\tau\in\uhp\colon\IM\tau\geq Y_0\}$ for any positive real number $Y_0$,
    \item $f_2$ is holomorphic and there is $n_0\in N^{-1}\Z$ such that $f_2=\sum_{n_0\leq n\in N^{-1}\Z}a_nq^n$.
\end{enumerate}
Then the holomorphic part of $f_1f_2$ is $H_1f_2$ and the nonholomorphic part is $H_2f_2$.
\end{prop}
\begin{proof}
Set $y=\IM\tau$. We write
\begin{equation*}
H_1(\tau)=\sum_{n\in N^{-1}\Z}c_nq^n,\qquad H_2(\tau)=\sum_{n\in N^{-1}\Z}d_n(y)q^n
\end{equation*}
according to Remark \ref{rema:holoPart}. Since $H_1$ and $f_2$ converge on $\abs{q}<1$, so is $H_1f_2$ by the theory of Laurent series. By assumption $H_2$ and by the theory of Laurent series $f_2$, converge normally on $\{\tau\in\uhp\colon\IM\tau\geq Y_0\}$. Therefore the double series
$$
H_2f_2=\sum_{n\in N^{-1}\Z}\left(\sum_{n_1+n_2=n}d_{n_1}(y)a_{n_2}\right)q^n
$$
converges normally. Set $b_n(y)=\sum_{n_1+n_2=n}d_{n_1}(y)a_{n_2}$. It remains to prove $\lim_{y\to+\infty}b_n(y)=0$. For each $n\in N^{-1}\Z$, define
\begin{equation*}
Y_{n}:=\sup\left\{\frac{n_1-n}{n_1}\colon n_1\in N^{-1}\Z,\,n_1<\min\{0,n\}\right\}.
\end{equation*}
Clearly, $0<Y_n<+\infty$. We split $b_n(y)$ into three parts:
\begin{equation*}
b_n(y)=\sum_{\substack{n_1+n_2=n\\0\leq n_1\leq n-n_0}}+\sum_{\substack{n_1+n_2=n\\n\leq n_1<0}}+\sum_{\substack{n_1+n_2=n\\n_1<\min\{0,n\}}}
\end{equation*}
and use $b_n^{(1)}(y)$, $b_n^{(2)}(y)$ and $b_n^{(3)}(y)$ to denote the three sums on the right-hand side. We want to show that we can interchange the limit $y\to+\infty$ and the summation $\sum_{n_1+n_2=n}$. The sums $b_n^{(1)}(y)$ and $b_n^{(2)}(y)$ are finite so let us consider $b_n^{(3)}(y)$. If $y\geq Y_n$, then for any $n_1\in N^{-1}\Z$ we have $e^{2\pi n_1y}\leq e^{-2\pi n_2}$ where $n_2=n-n_1$. It follows that for $y\geq Y_n$
\begin{align*}
\abs{b_n^{(3)}(y)}&\leq\sum_{n_1<\min\{0,n\}}\abs{d_{n_1}(y)a_{n_2}}=\sum_{n_1<\min\{0,n\}}\abs{d_{n_1}(y)}e^{-2\pi n_1y}\cdot\abs{a_{n_2}}e^{2\pi n_1y}\\
&\leq\sum_{n_1<\min\{0,n\}}\abs{d_{n_1}(y)}e^{-2\pi n_1y}\cdot\abs{a_{n_2}}e^{-2\pi n_2}\\
&\leq\sum_{n_1\in N^{-1}\Z}\abs{d_{n_1}(y)}e^{-2\pi n_1y}\cdot\sum_{n_2\in N^{-1}\Z}\abs{a_{n_2}}e^{-2\pi n_2}<+\infty
\end{align*}
since $H_2$, respectively $f_2$, converge normally on $\{\tau\in\uhp\colon\IM\tau\geq Y_n\}$, respectively $\{\tau\in\uhp\colon\IM\tau=1\}$.
Thus the series $b_n(y)$ converges uniformly on $y\geq Y_n$ and hence $\lim_{y\to+\infty}b_n(y)=0$ by interchanging limits and uniform limits for each $n$. This concludes the whole proof.
\end{proof}

The order $\ord_{\rmi\infty}\partial\widehat{A}_\ell(-x)\vert_2\gamma$, which is defined to be the exponent $n$ (possibly nonintegral) of the leading term $c_nq^n$ in the Fourier expansion of the holomorphic part of $\partial\widehat{A}_\ell(-x)\vert_2\gamma$, can be computed easily using Theorem \ref{thm:partAlhatCusp}:
\begin{cor}
\label{coro:orderPartialAlhat}
Recall that $M=\ell cx+\frac{\ell}{2}$. If
\begin{equation*}
(\ell+1)\cdot(\lceil-cx\rceil+cx)+\lceil M\rceil-M>\frac{\ell}{2}+1,
\end{equation*}
then we have
\begin{equation*}
\ord\nolimits_{\rmi\infty}\partial\widehat{A}_\ell(-x)\vert_2\gamma=\frac{\ell}{2}(\lceil-cx\rceil+cx-1)^2+(\lceil-cx\rceil+cx-1)\cdot(\lceil M\rceil-M-1).
\end{equation*}
On the other hand, if the reversed inequality holds, then we have
\begin{equation*}
\ord\nolimits_{\rmi\infty}\partial\widehat{A}_l(-x)\vert_2\gamma=\frac{\ell}{2}(\lceil-cx\rceil+cx)^2+(\lceil-cx\rceil+cx)\cdot(\lceil M\rceil-M).
\end{equation*}
\end{cor}
\begin{proof}
Recall that $H_1(\tau)$ in Theorem \ref{thm:partAlhatCusp} is the holomorphic part of $\frac{1}{2\uppi\rmi}\rme^{-\uppi\rmi \ell cdx^2}\cdot\partial\widehat{A}_\ell(-x)\vert_2\gamma$. Thus we need to find the leading term of $H_1(\tau)$. If $cx\in\Z$ then $dx\not\in\Z$ and hence $C\neq0$. Since the exponent of $q$ of any term in the sum $\sum_{n>-cx,\,m\geq M}-\sum_{n<-cx,\,m< M}$ is positive, we have $\ord\nolimits_{\rmi\infty}\partial\widehat{A}_\ell(-x)\vert_2\gamma=0$ and hence the required assertion holds in this case. Now assume that $cx\not\in\Z$. Then $C=0$. The exponent of $q$ of the leading term of the sum $\sum_{n>-cx,\,m\geq M}$ is equal to
$$
e_1=\frac{\ell}{2}(\lceil-cx\rceil+cx)^2+(\lceil-cx\rceil+cx)\cdot(\lceil M\rceil-M)
$$
and that of $\sum_{n<-cx,\,m< M}$ is equal to
$$
e_2=\frac{\ell}{2}(\lceil-cx\rceil+cx-1)^2+(\lceil-cx\rceil+cx-1)\cdot(\lceil M\rceil-M-1).
$$
Therefore $\ord\nolimits_{\rmi\infty}\partial\widehat{A}_l(-x)\vert_2\gamma=\min\{e_1,e_2\}$ since $e_1\neq e_2$. It remains to compare $e_1$ and $e_2$. Since
$$
e_1-e_2=(\ell+1)(\lceil-cx\rceil+cx)+\lceil M\rceil-M-(\frac{\ell}{2}+1)
$$
the desired assertion holds.
\end{proof}
Note that the case $(\ell+1)\cdot(\lceil-cx\rceil+cx)+\lceil M\rceil-M=\frac{\ell}{2}+1$ could not happen.

For any coprime integers $a,\,c$, we define $\ord_{a/c}f=\ord_{\rmi\infty}f\vert_2\tbtmat{a}{b}{c}{d}$, where $b,\,d$ are any integers satisfying $ad-bc=1$. This is well defined provided that $f$ satisfies the modular transformation equations on a subgroup of $\slZ$ of finite index. Thus Corollary \ref{coro:orderPartialAlhat} amounts to giving $\ord_{a/c}\partial\widehat{A}_l(-x)$ which, as one can easily see, is independent of $a$.

In the remainder we need another function related to $\partial\widehat{A}_\ell$, namely, the function
\begin{equation}
\label{eq:defglx}
g_{\ell,x}(\tau):=q^{-\frac{\ell}{2}x^2}\left(\partial\widehat{A}_\ell(x\tau;\tau)-2\pi i\ell x\widehat{A}_\ell(x\tau;\tau)\right).
\end{equation}
As before, $x$ is a nonintegral positive rational number and $k$ is a positive integer such that $kx\in\Z$.
\begin{prop}
\label{prop:glxModular}
We have $g_{\ell,x}=\partial\widehat{A}_l(-x)\vert_2\tbtmat{0}{-1}{1}{0}$. As a consequence, for any $\tbtmat{a}{b}{c}{d}\in\Gamma_1(k)$ we have
\begin{equation*}
g_{\ell,x}\left(k\frac{a\tau+b}{c\tau+d}\right)=(c\tau+d)^2e^{-\pi i\ell abkx^2}\cdot(-1)^{\ell((a-1)x+kbx)}g_{\ell,x}(k\tau).
\end{equation*}
\end{prop}
\begin{proof}
As in the proof of Theorem \ref{thm:partAlhatCusp}, we shall freely use the theory of Jacobi forms here. We can express $g_{\ell,x}$ in terms of $\widehat{A}_\ell$ as
\begin{equation}
\label{eq:glxModulareq1}
g_{\ell,x}(\tau)=(\partial\widehat{A}_\ell-2\pi i\ell x \widehat{A}_\ell)\vert_{2,-\ell/2}[(x,0),1](0;\tau).
\end{equation}
Setting $u=0$ and $\tbtmat{a}{b}{c}{d}=\tbtmat{0}{1}{-1}{0}$ in \eqref{eq:pAlvertxabcd} we find that
\begin{equation}
\label{eq:glxModulareq2}
\partial\widehat{A}_\ell(-x;\tau)\vert_2\tbtmat{0}{-1}{1}{0}=\partial\widehat{A}_\ell(-x;\tau)\vert_2\tbtmat{0}{1}{-1}{0}=\partial(\widehat{A}_\ell\vert_{1,-\ell/2}[(x,0),1])(0;\tau).
\end{equation}
By the chain rule for $\frac{\partial}{\partial u}$,
\begin{equation}
\label{eq:glxModulareq3}
\partial(\widehat{A}_\ell\vert_{1,-\ell/2}[(x,0),1])
=-2\pi i\ell x\widehat{A}_\ell\vert_{2,-\ell/2}[(x,0),1]+(\partial\widehat{A}_\ell)\vert_{2,-\ell/2}[(x,0),1]).
\end{equation}
Combining \eqref{eq:glxModulareq1}, \eqref{eq:glxModulareq2} and \eqref{eq:glxModulareq3} we obtain $g_{\ell,x}=\partial\widehat{A}_l(-x)\vert_2\tbtmat{0}{-1}{1}{0}$ as required. To prove the last assertion, let $\tbtmat{a}{b}{c}{d}\in\Gamma_1(k)$ be arbitrary. Then
\begin{align*}
(c\tau+d)^{-2}g_{\ell,x}\left(k\frac{a\tau+b}{c\tau+d}\right)&=(c\tau+d)^{-2}\left(k\frac{a\tau+b}{c\tau+d}\right)^{-2}\partial\widehat{A}_{\ell}\left(-x;-\frac{c\tau+d}{k(a\tau+b)}\right)\\
&=\partial\widehat{A}_{\ell}(-x)\vert_2\tbtmat{-c/k}{-d}{a}{bk}(k\tau)\\
&=\partial\widehat{A}_{\ell}(-x)\vert_2\tbtmat{d}{-c/k}{-bk}{a}\vert_2\tbtmat{0}{-1}{1}{0}(k\tau)\\
&=(-1)^{\ell x(a-bk-1)}e^{-\pi i\ell abkx^2}g_{\ell,x}(k\tau),
\end{align*}
where we have used \eqref{eq:transformationGamma1k} in the last equality.
\end{proof}

The holomorphic part and nonholomorphic part of $g_{\ell,x}(k\tau)$ at any cusp can be derived from the above proposition and Theorem \ref{thm:partAlhatCusp} of which we omit the details. What we need below are formulas for the orders of $g_{\ell,x}(k\tau)$ at cusps:
\begin{cor}
\label{cor:glxOrder}
Let $a,c$ be nonnegative coprime integers. Set $c'=\frac{ka}{(ka,c)}$ and $M'=\ell c'x+\frac{\ell}{2}$. If
\begin{equation*}
(\ell+1)\cdot(\lceil-c'x\rceil+c'x)+\lceil M'\rceil-M'>\frac{\ell}{2}+1,
\end{equation*}
then
\begin{equation*}
\ord\nolimits_{\frac{a}{c}}(g_{\ell,x}(k\tau))=\frac{(ka,c)^2}{k}\left(\frac{\ell}{2}(\lceil-c'x\rceil+c'x-1)^2+(\lceil-c'x\rceil+c'x-1)\cdot(\lceil M'\rceil-M'-1)\right).
\end{equation*}
On the other hand, if the reversed inequality holds, then
\begin{equation*}
\ord\nolimits_{\frac{a}{c}}(g_{\ell,x}(k\tau))=\frac{(ka,c)^2}{k}\left(\frac{\ell}{2}(\lceil-c'x\rceil+c'x)^2+(\lceil-c'x\rceil+c'x)\cdot(\lceil M'\rceil-M')\right).
\end{equation*}
\end{cor}
Before giving the proof, we extend the definition of slash operators a bit. Let $r\in\Z$. For a real matrix $\tbtmat{a}{b}{c}{d}$ with positive determinant and a function $f$ on $\uhp$, set $f\vert_r\tbtmat{a}{b}{c}{d}(\tau)=(ad-bc)^{r/2}\cdot(c\tau+d)^{-r}f\left(\frac{a\tau+b}{c\tau+d}\right)$. One can easily check that $f\vert_r(\gamma_1\gamma_2)=(f\vert_r\gamma_1)\vert_r\gamma_2$ for matrices $\gamma_1$ and $\gamma_2$ of positive determinants and that $g_{\ell,x}(k\tau)=k^{-1}g_{\ell,x}\vert_2\tbtmat{k}{0}{0}{1}(\tau)$.
\begin{proof}[Proof of Corollary \ref{cor:glxOrder}]
Let $a',b'$ be integers satisfying $(ka,c)=kaa'+cb'$ and $b,d$ be integers satisfying $ad-bc=1$. Then
\begin{equation}
\label{eq:MATk001abcd}
\MAT{0}{-1}{1}{0}\MAT{k}{0}{0}{1}\MAT{a}{b}{c}{d}=\MAT{-\frac{c}{(ka,c)}}{-a'}{\frac{ka}{(ka,c)}}{-b'}\MAT{(ka,c)}{ka'b+b'd}{0}{\frac{k}{(ka,c)}}
\end{equation}
with the first factor on the right-hand side in $\slZ$.
Thus, by Proposition \ref{prop:glxModular} we have
\begin{align*}
\ord\nolimits_{\frac{a}{c}}(g_{\ell,x}(k\tau))&=\ord\nolimits_{i\infty}(g_{\ell,x}\vert_2\tbtmat{k}{0}{0}{1}\vert_2\tbtmat{a}{b}{c}{d})\\
&=\ord\nolimits_{i\infty}(\partial\widehat{A}_\ell(-x)\vert_2\tbtmat{0}{-1}{1}{0}\tbtmat{k}{0}{0}{1}\tbtmat{a}{b}{c}{d})\\
&=\ord\nolimits_{i\infty}\left(\partial\widehat{A}_\ell(-x)\vert_2\tbtmat{-\frac{c}{(ka,c)}}{-a'}{\frac{ka}{(ka,c)}}{-b'}\tbtmat{(ka,c)}{ka'b+b'd}{0}{\frac{k}{(ka,c)}}\right)\\
&=\frac{(ka,c)^2}{k}\cdot\ord\nolimits_{-\frac{c}{ka}}\partial\widehat{A}_\ell(-x).
\end{align*}
The assertions follow from this and Corollary \ref{coro:orderPartialAlhat}.
\end{proof}

\subsection{The $U_{p,k}$ operator and its variant}

Let $p$ be a positive number and $k$ be an integer. The operator $U_{p,k}$ is given by
$$
U_{p,k}(f(\tau)):=\frac{1}{p}\sum_{m=0}^{p-1}\zeta_{p}^{-mk}f\left(\frac{\tau+m}{p}\right),
$$
where $f(\tau)$ is a function defined on $\uhp$.
If $f(\tau)$ has the Fourier expansion
$$
f(\tau)=\sum_{n\geq n_0}a(n)q^n,
$$
then one can easily verify
\begin{align}
\label{up}
U_{p,k}(g(q^p)f(\tau))=q^{\frac{k}{p}}g(q)\sum_{n\geq n_0-\left[\frac{k}{p}\right]}a(pn+k)q^n
\end{align}
for any $q$-series
$$
g(q)=\sum_{m\geq m_0}b(m)q^m.
$$

The following lemma relates the transformation equations of $U_{p,k}f$ to that of $f$. Let $\Gamma(p)$ denote the principal congruence subgroup of level $p$ consisting of integral modular matrices $\tbtmat{a}{b}{c}{d}$ with $a\equiv d\equiv1\pmod{p}$ and $b\equiv c\equiv0\pmod{p}$.
\begin{lemma}
\label{lemm:UpkTransformation}
Let $r$ be an even integer. Suppose $f$ satisfies $f\vert_r\gamma=\chi(\gamma)f$ for any $\gamma\in\Gamma_1(p)$ where $\chi$ is a linear character of $\Gamma_1(p)$ such that $\chi\tbtmat{a}{b/p}{cp}{d}=1$ for any $\MAT{a}{b}{c}{d}\in\Gamma(p)$. Then
\begin{equation*}
(U_{p,k}f)\vert_r\tbtmat{a}{b}{c}{d}=\zeta_p^{bk}\cdot U_{p,k}f,\quad\tbtmat{a}{b}{c}{d}\in\Gamma_1(p).
\end{equation*}
\end{lemma}
\begin{proof}
As in the proof of \cite[Lemma 2.8]{Mao-24}, for $\tbtmat{a}{b}{c}{d}\in\Gamma_1(p)$,
\begin{equation}
\label{eq:MAT1m0pabcd}
\MAT{1}{m}{0}{p}\MAT{a}{b}{c}{d}=\MAT{a+cm}{\frac{(d-a)m+b(1-a)-cm(m+b)}{p}}{pc}{-c(m+b)+d}\MAT{1}{m+b}{0}{p},
\end{equation}
with the second factor on the right-hand side, denoted by $\gamma_m$ here, in $\Gamma_1(p)$. It follows directly from the assumption that $\chi(\gamma_m)=1$ since $\tbtmat{1}{0}{0}{p}^{-1}\gamma_m\tbtmat{1}{0}{0}{p}\in\Gamma(p)$. Therefore,
\begin{align*}
(U_{p,k}f)\vert_r\tbtmat{a}{b}{c}{d}&=p^{\frac{r}{2}-1}\sum_{0\leq m<p}\zeta_p^{-mk}f\vert_r\tbtmat{1}{m}{0}{p}\tbtmat{a}{b}{c}{d}\\
&=p^{\frac{r}{2}-1}\sum_{0\leq m<p}\zeta_p^{-mk}f\vert_r\gamma_m\vert_r\tbtmat{1}{m+b}{0}{p}\\
&=\zeta_p^{bk}p^{\frac{r}{2}-1}\sum_{0\leq m<p}\zeta_p^{-(m+b)k}\chi(\gamma_m)f\vert_r\tbtmat{1}{m+b}{0}{p}\\
&=\zeta_p^{bk}\cdot U_{p,k}f.
\end{align*}
The last equality relies on the fact that $\zeta_p^{-(m+b)k}f\vert_r\tbtmat{1}{m+b}{0}{p}=\zeta_p^{-(m+b-p)k}f\vert_r\tbtmat{1}{m+b-p}{0}{p}$ because $\chi\tbtmat{1}{1}{0}{1}=1$ by assumption and $\zeta_p^{pk}=1$.
\end{proof}
\begin{remark}
According to \eqref{eq:transformationGamma1k}, when $p$ is odd and $j\in\Z$ is not divisible by $p$, $f=\partial\widehat{A}_\ell(-j/p)$ satisfies the condition of the above lemma. Moreover, Lemma 2.8 of Mao \cite{Mao-24} is a special case of the above lemma. Our proof is essentially identical to Mao's.
\end{remark}

If $\chi$ is different from the one in the above lemma, it is still possible that $U_{p,k}f$ is modular on $\Gamma_1(p)$ provided that we modify $U_{p,k}$ slightly according to $\chi$. We make no attempt to present the most general definition but give the following one which is precisely what we need in the proof of the main theorem:
$$
U'_{p,k}(f(\tau)):=\frac{1}{p}\sum_{m=0}^{p-1}\zeta_{24p}^{-m(p-1)}\zeta_{p}^{-mk}f\left(\frac{\tau+m}{p}\right).
$$
One can check immediately that
\begin{equation}
\label{eq:UpkUpkprime}
U'_{p,k}(q^{\frac{p-1}{24}}(q^p;q^p)_\infty\cdot f)=q^{\frac{p-1}{24p}}(q;q)_\infty\cdot U_{p,k}f,
\end{equation}
so $U'_{p,k}$ is just a reformulation of $U_{p,k}$. However, when $f$ itself is not modular while $q^{-1/24}f$ is, then it is necessary to shift the focus from $U_{p,k}$ to $U'_{p,k}$ for proving modularity.
We now give the analogue of Lemma \ref{lemm:UpkTransformation}, which plays a key role in the proof of the main theorems, for $U'_{p,k}$ acting on certain concrete functions. Recall that the Dedekind eta function is defined by $\eta(\tau)=q^{1/24}(q;q)_\infty$.
\begin{lemma}
\label{lemm:UpkprimeTransformation}
Let $p$ be an odd positive integer not divisible by $3$ (say, $p$ is a prime $>3$) and $0\leq k < p$, $0<j<p$. Set
$$
f=U_{p,k}'\left(q^{\frac{p-1}{24}}\frac{(q^p;q^p)_\infty}{(q;q)_\infty}\partial\widehat{A}_\ell(-j/p)\right)=U_{p,k}'\left(\frac{\eta(p\tau)}{\eta(\tau)}\partial\widehat{A}_\ell(-j/p;\tau)\right).
$$
If there exists an integer $v$ and a positive integer $\ell_1$ such that $12\ell_1 v^2\equiv-24k+1\pmod{p}$ and $v\not\equiv0\pmod{p}$, then $f$ transforms like $\eta(\tau)\eta(p\tau)^{-1}g_{\ell_1,v/p}(p\tau)$ (c.f. \eqref{eq:defglx}), that is,
\begin{multline*}
(c\tau+d)^{-2}\eta\left(p\frac{a\tau+b}{c\tau+d}\right)\eta\left(\frac{a\tau+b}{c\tau+d}\right)^{-1}f\left(\frac{a\tau+b}{c\tau+d}\right)\\
=e^{-\frac{\pi i\ell_1 abv^2}{p}}\cdot(-1)^{\ell_1 v((a-1)/p+b)}\eta(p\tau)\eta(\tau)^{-1}f(\tau),\quad \MAT{a}{b}{c}{d}\in\Gamma_1(p).
\end{multline*}
In particular, if $-24k+1\equiv0\pmod{p}$, then $f$ transforms like $\eta(\tau)\eta(p\tau)^{-1}$.
\end{lemma}
\begin{proof}
It is possible to produce a proof similar to that of Lemma \ref{lemm:UpkTransformation} but such a proof is rather tedious. The most concise proof makes use of the concept of generalized double coset operators introduced in \cite[Section 3]{ZZ23}. Set $f_0(\tau)=\frac{\eta(p\tau)}{\eta(\tau)}\partial\widehat{A}_\ell(-j/p;\tau)$. Let $\chi_{1^{-1}p^1}$, $\chi_{1^{1}p^{-1}}$, $\phi_{\ell, j/p}$, $\psi_{\ell_1, v/p}$ be the characters\footnote{If a nonzero function $f$ satisfies $f\vert_r\gamma=\chi(\gamma)f$ for any $\gamma\in G$ where $\chi(\gamma)\in\C$, $G$ is a finite index subgroup of $\slZ$ and $r$ is an integer, then the character of $f$ is defined to be the map $\gamma\mapsto\chi(\gamma)$ on $G$ which turns out to be a complex linear group character.} of $\eta(\tau)^{-1}\eta(p\tau)$, $\eta(\tau)\eta(p\tau)^{-1}$, $\partial\widehat{A}_\ell(-j/p;\tau)$ and $g_{\ell_1,v/p}(p\tau)$ respectively. The formulas for $\phi_{\ell, j/p}$ and $\psi_{\ell_1, v/p}$ have been given in \eqref{eq:transformationGamma1k} and Proposition \ref{prop:glxModular} respectively. For the formulas for $\chi_{1^{-1}p^1}$, $\chi_{1^{1}p^{-1}}$ see \cite[Eq. (15) and (16)]{ZZ23} or \eqref{eq:etaChar} which is due to H. Petersson. One can verify that $\Gamma_1(p)\tbtmat{1}{0}{0}{p}\Gamma_1(p)=\cup_{0\leq m <p}\Gamma_1(p)\tbtmat{1}{0}{0}{p}\tbtmat{1}{m}{0}{1}$ which is a disjoint union using, for instance, \eqref{eq:MAT1m0pabcd}. Thus, if we can prove that
\begin{equation}
\label{eq:toProveCharEq}
\chi_{1^{-1}p^1}\MAT{a}{b/p}{cp}{d}\phi_{\ell, j/p}\MAT{a}{b/p}{cp}{d}=\chi_{1^{1}p^{-1}}\MAT{a}{b}{c}{d}\psi_{\ell_1, v/p}\MAT{a}{b}{c}{d}
\end{equation}
for $\tbtmat{a}{b}{c}{d}\in\tbtmat{1}{0}{0}{p}^{-1}\Gamma_1(p)\tbtmat{1}{0}{0}{p}\cap\Gamma_1(p)=\Gamma(p)$, then
$$
T_pf_0:=\frac{1}{p}\sum_{0\leq m<p}\chi_{1^{1}p^{-1}}\cdot\psi_{\ell_1, v/p}\MAT{1}{m}{0}{1}^{-1}f_0\left(\frac{\tau+m}{p}\right)
$$
transforms like $\eta(\tau)\eta(p\tau)^{-1}g_{\ell_1,v/p}(p\tau)$ according to \cite[Proposition 3.2(3) and (4)]{ZZ23}. Since $12\ell_1 v^2\equiv-24k+1\pmod{p}$ and $p^2-1\equiv0\pmod{24}$ we have $\chi_{1^{1}p^{-1}}\cdot\psi_{\ell_1, v/p}\tbtmat{1}{m}{0}{1}^{-1}=\zeta_{24p}^{-m(p-1)}\zeta_{p}^{-mk}$  which means $T_p=U'_{p,k}$. It remains to prove \eqref{eq:toProveCharEq}. Let $\chi_\eta$ be the multiplier system of $\eta(\tau)$ (c.f. \eqref{eq:etaChar}, the Petersson's formula). Then
\begin{align*}
\chi_{1^{-1}p^1}\tbtmat{a}{b/p}{cp}{d}&=\chi_{\eta(p\tau)}\tbtmat{a}{b/p}{cp}{d}\chi_{\eta}^{-1}\tbtmat{a}{b/p}{cp}{d}\\
&=\chi_{\eta}\tbtmat{a}{b}{c}{d}\chi_{\eta}^{-1}\tbtmat{a}{bp}{c/p}{d}\\
&=\chi_{\eta}\tbtmat{a}{b}{c}{d}\chi_{\eta(p\tau)}^{-1}\tbtmat{a}{b}{c}{d}\\
&=\chi_{1^{1}p^{-1}}\tbtmat{a}{b}{c}{d},
\end{align*}
where $\chi_{\eta}^{-1}\tbtmat{a}{b/p}{cp}{d}=\chi_{\eta}^{-1}\tbtmat{a}{bp}{c/p}{d}$ follows from the fact $24\mid p^2-1$ and \eqref{eq:etaChar}. Thus the desired equality \eqref{eq:toProveCharEq} is equivalent to $\phi_{\ell, j/p}\tbtmat{a}{b/p}{cp}{d}=\psi_{\ell_1, v/p}\tbtmat{a}{b}{c}{d}$ which actually holds since
\begin{align*}
\phi_{\ell, j/p}\tbtmat{a}{b/p}{cp}{d}&=(-1)^{\frac{\ell j(pc+d-1)}{p}}e^{\frac{\pi i\ell j^2cd}{p}}=(-1)^{\ell j\left(c+\frac{d-1}{p}+\frac{cd}{p}\right)}=1
\end{align*}
and
\begin{align*}
\psi_{\ell_1, v/p}\tbtmat{a}{b}{c}{d}&=(-1)^{\ell_1 v\left(\frac{a-1}{p}+b\right)}e^{\frac{-\pi i\ell_1 v^2ab}{p}}=(-1)^{\ell_1 v\left(b+\frac{a-1}{p}-\frac{ab}{p}\right)}=1.
\end{align*}
\end{proof}

We end this subsection with a lemma concerning the orders of $U'_{p,k}f$ at cusps.
\begin{lemma}
\label{lemma:ordUpkprime}
Let the notation be as in Lemma \ref{lemm:UpkprimeTransformation} and let $a,c$ be coprime integers. Set $f_0=q^{\frac{p-1}{24}}\frac{(q^p;q^p)_\infty}{(q;q)_\infty}\partial\widehat{A}_\ell(-j/p)$. We have
\begin{equation*}
\ord\nolimits_{\frac{a}{c}}U'_{p,k}f_0\geq\min_{0\leq m< p}\frac{\gcd(a+cm,cp)^2}{p}\cdot\ord\nolimits_{\frac{a+cm}{cp}}f_0.
\end{equation*}
\end{lemma}
\begin{proof}
Let $m\in\mathbb{Z}$ and $\tbtmat{a}{b}{c}{d}\in\slZ$. Set $g=\gcd(a+cm,cp)$. Then
\begin{equation*}
\MAT{1}{m}{0}{p}\MAT{a}{b}{c}{d}=\MAT{(a+cm)g^{-1}}{-y}{cpg^{-1}}{x}\MAT{g}{(b+dm)x+dpy}{0}{pg^{-1}},
\end{equation*}
where $x,y$ are integers with the property $(a+cm)x+cpy=g$. The first factor on the right-hand side is in $\slZ$ and the second is rational. Therefore
\begin{align*}
\ord\nolimits_{\frac{a}{c}}U'_{p,k}f_0&=\ord\nolimits_{\rmi\infty}(U'_{p,k}f_0)\vert_2\tbtmat{a}{b}{c}{d}\\
&\geq\min_{0\leq m< p}\ord\nolimits_{\rmi\infty}f_0\vert_2\tbtmat{1}{m}{0}{p}\vert_2\tbtmat{a}{b}{c}{d}\\
&=\min_{0\leq m< p}\ord\nolimits_{\rmi\infty}f_0\vert_2\tbtmat{(a+cm)g^{-1}}{-y}{cpg^{-1}}{x}\vert_2\tbtmat{g}{(b+dm)x+dpy}{0}{pg^{-1}}\\
&=\min_{0\leq m< p}\frac{\gcd(a+cm,cp)^2}{p}\cdot\ord\nolimits_{\frac{a+cm}{cp}}f_0.
\end{align*}
\end{proof}
\begin{remark}
\label{rema:ordUpkprime}
To calculate $\ord\nolimits_{\frac{a+cm}{cp}}f_0=\ord_{i\infty}f_0\vert_2\gamma$ where $\gamma\in\slZ$ such that $\gamma(i\infty)=\frac{a+cm}{cp}$, let $H_1$ be the holomorphic part of $\partial\widehat{A}_\ell(-j/p)\vert_2\gamma$. By Theorem \ref{thm:partAlhatCusp}, the series expression of $H_1$ and that of the nonholomorphic part converge normally on $\{\tau\in\uhp\colon\IM\tau\geq Y_0\}$ for any $Y_0>0$. Therefore, the holomorphic part of $f_0\vert_2\gamma$ is equal to $\eta(p\gamma\tau)\eta(\gamma\tau)^{-1}\cdot H_1(\tau)$ by Proposition \ref{prop:productHolo} and consequently, one can calculate $\ord\nolimits_{\frac{a+cm}{cp}}f_0$ using Corollary \ref{coro:orderPartialAlhat}. Finally, note that the fact $\ord\nolimits_{\frac{a}{c}}U'_{p,k}f_0$ is well defined follows from Corollary \ref{coro:Hpart} and Proposition \ref{prop:productHolo}.
\end{remark}

\subsection{Generalized Dedekind eta functions}
To cancel the multiplier system of $N_p(s,k)$ in Theorem \ref{mainN}, we need the generalized Dedekind eta functions. Let $p$ be a positive integer and $\delta\in\Z$. Set
\begin{equation}
\label{eq:gEta}
\eta_{p,\delta}(\tau)=q^{\frac{p}{2}P_2\left(\frac{\delta}{p}\right)}\cdot\prod_{\substack{n\equiv-\delta\bmod{p} \\ n\geq p-\delta}}(1-q^n)\prod_{\substack{n\equiv\delta\bmod{p} \\ n\geq\delta}}(1-q^n)=q^{\frac{p}{2}P_2\left(\frac{\delta}{p}\right)}\cdot(q^\delta,q^{p-\delta};q^p)_\infty,
\end{equation}
where $P_2(x)=x^2-x+\frac{1}{6}$ is the second Bernoulli polynomial. Since $\eta_{p,\delta}=-\eta_{p,\delta+p}$ we assume $0\leq\delta<p$ without loss of generality. When $p\mid\delta$ we have $\eta_{p,\delta}=0$ which differs from some authors' definition\footnote{We adopt the definition \eqref{eq:gEta} since they are specializations of the weight $1/2$, index $1/2$ Jacobi form $\theta(z;\tau)$ uniformly for all $p$ and $\delta$.}. It is known that $\eta_{p,\delta}$ is a modular function on $\Gamma_1(p)$ and is holomorphic on $\uhp$, possibly with a multiplier system (character) which is denoted by $\chi_{\eta,p,\delta}$ hereafter (c.f. \cite{Rob-94}). S. Robins \cite{Rob-94} also gives the order of $\eta_{p,\delta}$ (for $p\nmid\delta$) at the cusp $\frac{a}{c}$, namely,
\begin{equation}
\label{eq:ordGEta}
\ord\nolimits_{\frac{a}{c}}\eta_{p,\delta}=\frac{(p,c)^2}{2p}\overline{P}_2\left(\frac{a\delta}{(p,c)}\right),
\end{equation}
where $\overline{P}_2(x)=P_2(x-[x])$ is the second Bernoulli function. On the other hand, although there are known formulas for $\chi_{\eta,p,\delta}$ expressed in terms of Meyer sums (c.f. \cite{Mey-57}), we need an alternative exact formula which is expressed in terms of Petersson's formula for the multiplier system $\chi_\eta$ of Dedekind eta function:
\begin{equation}
\label{eq:etaChar}
\chi_\eta\MAT{a}{b}{c}{d}=\begin{cases}
\leg{d}{\abs{c}}\exp{\frac{\pi i}{12}\left((a+d-3)c-bd(c^2-1)\right)}   & \text{if }2 \nmid c, \\
\leg{c}{d}\exp{\frac{\pi i}{12}\left((a-2d)c-bd(c^2-1)+3d-3\right)}   & \text{if }2 \mid c,
\end{cases}
\end{equation}
where $\leg{\cdot}{\cdot}$ is the Kronecker-Jacobi symbol. For a proof of \eqref{eq:etaChar}, see \cite{Kno-70}.
\begin{lemma}
\label{eq:charGEta}
Let $\tbtmat{a}{b}{c}{d}\in\Gamma_1(p)$. Then
\begin{equation*}
\chi_{\eta,p,\delta}\MAT{a}{b}{c}{d}:=\left.\eta_{p,\delta}\left(\frac{a\tau+b}{c\tau+d}\right)\middle/\eta_{p,\delta}(\tau)\right.=\chi_\eta^2\MAT{a}{pb}{c/p}{d}e^{\frac{\pi iab\delta^2}{p}}(-1)^{\frac{(a-1)\delta}{p}+b\delta}
\end{equation*}
\end{lemma}
\begin{proof}
Set $f_{\delta/p}(\tau)=q^{\frac{\delta^2}{2p^2}}\theta(\frac{\delta}{p}\tau;\tau)$. We have $f_{\delta/p}(p\tau)=-i\eta(p\tau)\eta_{p,\delta}(\tau)$ by \eqref{eq:triProd}. The required assertion will follow from the well known modular transformation equations of $\theta$:
\begin{align*}
\theta\left(\frac{z}{c\tau+d};\frac{a\tau+b}{c\tau+d}\right)&=(c\tau+d)^{1/2}e^{\pi i\frac{c}{c\tau+d}z^2}\cdot\chi_\eta^3\MAT{a}{b}{c}{d}\theta(z;\tau),\quad\MAT{a}{b}{c}{d}\in\slZ,\\
\theta(z+\lambda\tau+\mu;\tau)&=q^{-\frac{1}{2}\lambda^2}e^{-2\pi i \lambda z}(-1)^{\lambda+\mu}\theta(z;\tau),\quad\lambda,\mu\in\Z,
\end{align*}
which can be proved, for instance, using \cite[Lemma 2.2]{BFOR-17} and the facts $\chi_\eta\tbtmat{1}{1}{0}{1}=\zeta_{24}$ and $\chi_\eta\tbtmat{0}{-1}{1}{0}=\zeta_8^{-1}$. Note that $z^{1/2}$ must be understood as the principal branch, that is, $z^{1/2}=\exp\frac{1}{2}\log z$ with $-\pi <\IM\log z \leq \pi$. It follows that, for $\tbtmat{a}{b}{c}{d}\in\Gamma_1(p)$,
\begin{align*}
f_{\delta/p}&\left(\frac{a\tau+bp}{cp^{-1}\tau+d}\right)=(cp^{-1}\tau+d)^{1/2}e^{\frac{\pi i\delta^2(a\tau+bp)a}{p^2}}\chi_\eta^3\tbtmat{a}{bp}{cp^{-1}}{d}\theta\left(\frac{\delta}{p}\tau+\frac{a-1}{p}\delta\tau+b\delta;\tau\right)\\
&=(cp^{-1}\tau+d)^{1/2}e^{\frac{\pi i\delta^2(a\tau+bp)a}{p^2}}\chi_\eta^3\tbtmat{a}{bp}{cp^{-1}}{d}q^{-\frac{1}{2}\left(\frac{a-1}{p}\delta\right)^2}e^{-2\pi i\frac{(a-1)\delta^2\tau}{p^2}}(-1)^{\frac{a-1}{p}\delta+b\delta}\theta\left(\frac{\delta}{p}\tau;\tau\right)\\
&=(cp^{-1}\tau+d)^{1/2}\chi_\eta^3\tbtmat{a}{bp}{cp^{-1}}{d}e^{\frac{\pi i ab\delta^2}{p}}(-1)^{\frac{a-1}{p}\delta+b\delta}f_{\delta/p}(\tau).
\end{align*}
Therefore,
\begin{align*}
\eta_{p,\delta}\vert_0\tbtmat{a}{b}{c}{d}&=i\cdot(f_{\delta/p}\eta^{-1})\vert_0\tbtmat{p}{0}{0}{1}\tbtmat{a}{b}{c}{d}\\
&=i\cdot(f_{\delta/p}\eta^{-1})\vert_0\tbtmat{a}{bp}{c/p}{d}\tbtmat{p}{0}{0}{1}\\
&=\chi_\eta^2\tbtmat{a}{bp}{cp^{-1}}{d}e^{\frac{\pi i ab\delta^2}{p}}(-1)^{\frac{a-1}{p}\delta+b\delta}\cdot i(f_{\delta/p}\eta^{-1})\vert_0\tbtmat{p}{0}{0}{1}\\
&=\chi_\eta^2\tbtmat{a}{bp}{cp^{-1}}{d}e^{\frac{\pi i ab\delta^2}{p}}(-1)^{\frac{a-1}{p}\delta+b\delta}\eta_{p,\delta},
\end{align*}
which concludes the proof.
\end{proof}

Combining the techniques developed above in this section we obtain:
\begin{theorem}
\label{thm:AppellGEtaModular}
Let $p$ be an odd positive integer not divisible by $3$ and $0\leq k < p$, $0<j<p$. Let $r_1,\dots,r_{(p-1)/2}$ be integers. If
\begin{equation}
\label{eq:rvrelation}
12\cdot\sum_{\delta=1}^{(p-1)/2}r_\delta\cdot \delta^2\equiv -24k+1 \pmod{p},
\end{equation}
then the real analytic function
\begin{equation*}
f:=U_{p,k}\left(\frac{\partial\widehat{A}_\ell(-j/p)}{(q;q)_\infty}\right)\cdot q^{n_0}(q^p;q^p)_\infty^{1-2\sum r_\delta}\cdot\prod_{\delta=1}^{(p-1)/2}(q^\delta,q^{p-\delta};q^p)_\infty^{r_\delta},
\end{equation*}
where
\begin{equation*}
n_0=\frac{p^2-1}{24p}+\frac{1}{2p}\sum_{\delta=1}^{(p-1)/2}r_\delta\cdot\left(\delta^2-p\delta\right),
\end{equation*}
satisfies that
\begin{equation}
\label{eq:UpkGEtaModular}
f\left(\frac{a\tau+b}{c\tau+d}\right)=(c\tau+d)^{2-\sum\nolimits_\delta r_\delta}\cdot f(\tau),\quad \MAT{a}{b}{c}{d}\in\Gamma_1(p),
\end{equation}
and vice versa.
\end{theorem}
\begin{proof}
By \eqref{eq:UpkUpkprime} and \eqref{eq:gEta} we can express $f$ in terms of generalized Dedekind eta functions as $f=f_1f_2f_3$ where
\begin{align*}
f_1(\tau)&=\frac{\eta(p\tau)}{\eta(\tau)}U_{p,k}'\left(\frac{\eta(p\tau)}{\eta(\tau)}\partial\widehat{A}_\ell(-j/p;\tau)\right),\\
f_2(\tau)&=\eta(p\tau)^{-2\sum_\delta r_\delta},\\
f_3(\tau)&=\prod_{\delta=1}^{(p-1)/2}\eta_{p,\delta}(\tau)^{r_\delta}.
\end{align*}
Let $\tbtmat{a}{b}{c}{d}\in\Gamma_1(p)$ be arbitrary. By Lemma \ref{lemm:UpkprimeTransformation} we have
\begin{equation*}
(c\tau+d)^{-2}\cdot f_1\left(\frac{a\tau+b}{c\tau+d}\right)=e^{-\frac{\pi i\ell_1 abv^2}{p}}\cdot(-1)^{\ell_1 v((a-1)/p+b)}f_1(\tau),
\end{equation*}
where $l_1$ and $v$ are any integers with $\ell_1>0$, $12\ell_1v^2\equiv-24k+1\pmod{p}$ and $v\not\equiv0\pmod{p}$.
By the definition of $\chi_\eta$ we have
\begin{equation*}
(c\tau+d)^{\sum_\delta r_\delta}\cdot f_2\left(\frac{a\tau+b}{c\tau+d}\right)=\chi_\eta^{-2\sum_\delta r_\delta}\MAT{a}{pb}{c/p}{d}f_2(\tau).
\end{equation*}
According to Lemma \ref{eq:charGEta}, we have
\begin{equation*}
f_3\left(\frac{a\tau+b}{c\tau+d}\right)=\chi_\eta^{2\sum_\delta r_\delta}\MAT{a}{pb}{c/p}{d}\prod_{\delta=1}^{(p-1)/2}\left(e^{\frac{\pi iabr_\delta\delta^2}{p}}(-1)^{((a-1)/p+b)r_\delta\delta}\right)f_3(\tau).
\end{equation*}
Taking the product of these three identities we find that \eqref{eq:UpkGEtaModular} is equivalent to
\begin{equation*}
e^{-\frac{\pi i\ell_1 abv^2}{p}}\cdot(-1)^{\ell_1 v((a-1)/p+b)}\prod_{\delta=1}^{(p-1)/2}\left(e^{\frac{\pi iabr_\delta\delta^2}{p}}(-1)^{((a-1)/p+b)r_\delta\delta}\right)=1,
\end{equation*}
that is,
\begin{equation}
\label{eq:toProveEquiv}
\frac{ab}{p}\left(\sum_{\delta=1}^{(p-1)/2}r_\delta\delta^2-\ell_1 v^2\right)+\left(\frac{a-1}{p}+b\right)\left(\sum_{\delta=1}^{(p-1)/2}r_\delta\delta+\ell_1 v\right)\in2\Z
\end{equation}
for any $\tbtmat{a}{b}{c}{d}\in\Gamma_1(p)$. Now assume \eqref{eq:rvrelation} holds; then $\sum_\delta r_\delta\delta^2\equiv \ell_1 v^2 \pmod{p}$. If $\sum_\delta r_\delta\delta\equiv \ell_1 v \pmod{2}$, then $\sum_{\delta}r_\delta\delta^2\equiv \ell_1 v^2 \pmod{2p}$ from which \eqref{eq:toProveEquiv} follows. Otherwise, if $\sum_\delta r_\delta\delta\equiv \ell_1 v+1 \pmod{2}$, then $\sum_{\delta}r_\delta\delta^2\equiv \ell_1 v^2+p \pmod{2p}$ in which case \eqref{eq:toProveEquiv} is equivalent to
\begin{equation}
\label{eq:toProveEquiv2}
ab+\frac{a-1}{p}+b\in2\Z,\quad\tbtmat{a}{b}{c}{d}\in\Gamma_1(p).
\end{equation}
If $\frac{a-1}{p}$ is odd, then $a$ is even and hence $b$ is odd in which case \eqref{eq:toProveEquiv2} holds. If $\frac{a-1}{p}$ is even, then $a$ is odd in which case \eqref{eq:toProveEquiv2} holds as well whenever $b$ is odd or even. We have proved that \eqref{eq:rvrelation} implies \eqref{eq:UpkGEtaModular}. To prove the converse, just set $\tbtmat{a}{b}{c}{d}=\tbtmat{1}{1}{0}{1}$ in \eqref{eq:toProveEquiv}.
\end{proof}

\section{Generating functions}
\label{sec:Generating functions}

\subsection{Mock modular part of $NT$}

We consider the case $p\geq 5$ being prime and rewrite \eqref{maoid} as follows.

\begin{lemma}
\label{mao24p}
For prime $p\geq 5$ and $0<s<\frac{p}{2}$. We have
\begin{align}
\label{maoidp}
\sum_{n=0}^\infty\left(NT(s,p,n)-NT(p-s,p,n)\right)q^n=\mathcal{F}_{p,s}(\tau)+\sum_{r=1}^{p-1}\frac{p-2r}{2p}D(r-s,p),
\end{align}
where
$$
\mathcal{F}_{p,s}(\tau)=\frac{\delta_{s,1}}{2}+\sum_{j=1}^{p-1}\frac{\zeta_{p}^{j(s-1/2)}(1-\zeta_{p}^j)}{2p\pi i (q)_\infty}\cdot \frac{\partial}{\partial u}\bigg|_{u=0} A_3\left(u-\frac{j}{p};\tau\right).
$$
\end{lemma}

\begin{proof}
One can check that
\beq
\label{a1}
\sum_{j=1}^{p-1}\zeta_{p}^{j(s-1)}(1-\zeta_{p}^j)=p\delta_{s,1},
\eeq
since $0<s<p/2$.
Let
$$
g_k=\sum_{j=1}^{p-1}\frac{\zeta_{p}^{jk}}{1-\zeta_{p}^j}.
$$
If $k\equiv 0 \pmod{p}$ then
$$
g_k=\sum_{j=1}^{p-1}\frac{1}{1-\zeta_{p}^j}=\sum_{j=1}^{(p-1)/2}\left(\frac{1}{1-\zeta_{p}^j}+\frac{1}{1-\zeta_{p}^{-j}}\right)=\frac{p-1}{2}.
$$
For $0<k<p$ we have
$$
g_k-g_{k+1}=\sum_{j=1}^{p-1}\left(\frac{\zeta_{p}^{jk}}{1-\zeta_{p}^j}-\frac{\zeta_{p}^{jk+j}}{1-\zeta_{p}^j}\right)
=\sum_{j=1}^{p-1}\zeta_{p}^{jk}=-1.
$$
Hence $g_k=k-(p+1)/2$ if $0<k\leq p$ and $g_0=g_p=(p-1)/2$. By
$$
\mathcal{R}\left(\zeta_{p}^j;q\right)=\sum_{n=0}^\infty\sum_{m=0}^{p-1}\zeta_{p}^{mj}N(m,p,n)q^n,
$$
and noting that $N(r,p,n)=N(r+p,p,n)$, we have
\begin{align}
\label{a2}
-\sum_{j=1}^{p-1}\frac{\zeta_{p}^{js}(1+\zeta_{p}^j)\mathcal{R}\left(\zeta_{p}^j;q\right)}{(1-\zeta_{p}^j)}=
&-\sum_{j=1}^{p-1}\frac{\zeta_{p}^{js}(1+\zeta_{p}^j)}{1-\zeta_{p}^j}\sum_{n=0}^\infty\sum_{r=0}^{p-1}\zeta_{p}^{rj}N(r,p,n)q^n\\
\nonumber
=&-\sum_{n=0}^\infty\sum_{r=0}^{p-1}(g_{r}+g_{r+1})N(r-s,p,n)q^n\\
\nonumber
=&\sum_{r=1}^{p-1}(p-2r)\sum_{n=0}^\infty N(r-s,p,n)q^n\\
\nonumber
=&\sum_{r=1}^{p-1}(p-2r)D(r-s,p).
\end{align}
Substituting \eqref{a1} and \eqref{a2} to \eqref{maoid} we have \eqref{maoidp}.
\end{proof}

The part $\mathcal{F}_{p,s}(\tau)$ is a weight $3/2$ mock modular form. The completed form is given by
\begin{equation}
\label{eq:defHatF}
\widehat{\mathcal{F}}_{p,s}(\tau):=\sum_{j=1}^{p-1}\frac{\zeta_{p}^{j(s-1/2)}(1-\zeta_{p}^j)}{2p\pi i (q)_\infty}\cdot \frac{\partial}{\partial u}\bigg|_{u=0} \widehat{A}_3\left(u-\frac{j}{p};\tau\right).
\end{equation}
Following \cite[Lemma 2.9]{Mao-24}, define
\begin{align*}
T_{m,k}(q):=\sum_{n=-\infty}^\infty &\left[\sgn\left(\frac{1}{2}+kn+m\right)-E\left(\left(kn+m+\frac{1}{6}\right)\sqrt{\frac{6y}{k}}\right)\right]\\
&\times (-1)^n(6kn+6m+1)q^{\frac{-3kn^2+(6m+1)n}{2}},
\end{align*}
and
\begin{align*}
S_{m,k}(q):=\sum_{n=-\infty}^\infty &\left[\sgn\left(\frac{1}{2}-kn-m\right)+E\left(\left(kn+m+\frac{1}{6}\right)\sqrt{\frac{6y}{k}}\right)\right]\\
&\times (-1)^n(6kn+6m+1)q^{\frac{-3kn^2+(6m+1)n}{2}},
\end{align*}
with $y=\mathrm{Im(\tau)}$. We mention that $T_{m,k}$ and $S_{m,k}$ can not be expanded into a normal q-series so that we can not use \eqref{up} directly. But similarly, by a straightforward calculation, we have
\begin{align}
\label{upt}
U_{p,k}(q^nT_{m,p}(q^p))=&q^{n/p}T_{m,p}(q),\\
\label{ups}
U_{p,k}(q^nS_{m,p}(q^p))=&q^{n/p}S_{m,p}(q),
\end{align}
if $n\equiv k \pmod{p}$ and
\beq
\label{upts}
U_{p,k}(q^nT_{m,p}(q^p))=U_{p,k}(q^nS_{m,p}(q^p))=0,
\eeq
if $n\not\equiv k \pmod{p}$.

\begin{lemma}
\label{lmuf}
For prime $p\geq 5$, the completed form of $U_{p,k}\left(\mathcal{F}_{p,s}(\tau)\right)$ is given by the following.

(1)If $6k+s^2-s\equiv 0 \pmod{p}$ then
$$
U_{p,k}\left(\widehat{\mathcal{F}}_{p,s}(\tau)\right)=U_{p,k}\left(\mathcal{F}_{p,s}(\tau)\right)+\frac{1}{4}R_{p,s-1}-\frac{\delta_{s,1}}{2},
$$

(2)If $6k+s^2+s\equiv 0 \pmod{p}$ then
$$
U_{p,k}\left(\widehat{\mathcal{F}}_{p,s}(\tau)\right)=U_{p,k}\left(\mathcal{F}_{p,s}(\tau)\right)-\frac{1}{4}R_{p,s},
$$

(3)If $6k+s^2\pm s\not\equiv 0 \pmod{p}$ then
$$
U_{p,k}\left(\widehat{\mathcal{F}}_{p,s}(\tau)\right)=U_{p,k}\left(\mathcal{F}_{p,s}(\tau)\right),
$$
where
$$
R_{p,n}:=(-1)^{m_1}q^{-\frac{3m_1^2+m_1}{2p}}S_{m_1,p}(q)-(-1)^{m_2}q^{-\frac{3m_2^2+m_2}{2p}}T_{m_2,p}(q),
$$
with $n\equiv 3m_1 \pmod{p}$ and $n\equiv -3m_2-1 \pmod{p}$.
\end{lemma}

\begin{remark}
One can easily verify that $R_{p,n}$ does not depend on the choice of $m_1$ and $m_2$.
\end{remark}

\begin{proof}
Following the proof of \cite[Lemma 2.9]{Mao-24}, we have
\begin{align*}
&\widehat{A}_3\left(u-\frac{j}{k};\tau \right)-A_3\left(u-\frac{j}{k};\tau \right)\\
=&\frac{\pi i (q)_\infty}{2} \sum_{m=0}^{p-1}(-1)^mq^{-\frac{3m^2+m}{2}}
\left(\zeta_p^{-(6m+1)j/2}S_{m,p}(q^p)-\zeta_p^{(6m+1)j/2}T_{m,p}(q^p)\right).
\end{align*}
Hence
\begin{align*}
&\widehat{\mathcal{F}}_{p,s}(\tau)-\mathcal{F}_{p,s}(\tau)+\frac{\delta_{s,1}}{2}\\
=&\frac{1}{4p}\sum_{j=1}^{p-1} \zeta_p^{j(s-1/2)}(1-\zeta_p^j)\sum_{m=0}^{p-1} (-1)^mq^{-\frac{3m^2+m}{2}} \left(\zeta_p^{-(6m+1)j/2}S_{m,p}(q^p)-\zeta_p^{(6m+1)j/2}T_{m,p}(q^p)\right)\\
=&\frac{1}{4p}\sum_{m=0}^{p-1}(-1)^mq^{-\frac{3m^2+m}{2}}\sum_{j=1}^{p-1}\left(\zeta_p^{(s-3m-1)j}(1-\zeta_p^j)S_{m,p}(q^p)-\zeta_p^{(s+3m)j}(1-\zeta_p^j)T_{m,p}(q^p)\right)\\
=&\frac{1}{4}\big[(-1)^{m_1}q^{-\frac{3m_1^2+m_1}{2}}S_{m_1,p}(q^p)-(-1)^{m_2}q^{-\frac{3m_2^2+m_2}{2}}T_{m_2,p}(q^p)\\
&-(-1)^{n_1}q^{-\frac{3n_1^2+n_1}{2}}S_{n_1,p}(q^p)+(-1)^{n_2}q^{-\frac{3n_2^2+n_2}{2}}T_{n_2,p}(q^p) \big],
\end{align*}
where $s-3m_1-1\equiv 0 \pmod{p}$, $s+3m_2\equiv 0 \pmod{p}$, $s-3n_1-1\equiv -1 \pmod{p}$ and $s+3n_1\equiv -1 \pmod{p}$. The last equation holds by the following.
$$
\sum_{j=1}^{p-1}\zeta_p^{jk}(1-\zeta_p^j)
$$
is equal to $p$ if $k\equiv 0 \pmod{p}$, $-p$ if $k\equiv -1 \pmod{p}$ and 0 else. For $i=1,2$ we have
$$
-\frac{3m_i^2+m_i}{2}\equiv \frac{s^2-s}{6} \pmod{p},
$$
and
$$
-\frac{3n_i^2+n_i}{2}\equiv \frac{s^2+s}{6} \pmod{p}.
$$
Hence the Lemma holds via the fact \eqref{upt}-\eqref{upts}.
\end{proof}

To find certain q-series to cancel the nonholomorphic part of $\widehat{\mathcal{F}}_{p,s}(\tau)$, we let
$$
\mathcal{L}_p(v):=\frac{(-1)^vq^{\frac{-3v^2-2s_p}{2p}}}{(q^p;q^p)_\infty}\left(\frac{\partial}{\partial u}\bigg|_{u=0} \frac{p}{2\pi i}A_3(u+v\tau;p\tau)-3vA_3(v\tau;p\tau)\right),
$$
with integer $0<v<p$, and its completed form
$$
\widehat{\mathcal{L}}_p(v):=\frac{(-1)^vq^{\frac{-3v^2-2s_p}{2p}}}{(q^p;q^p)_\infty}\left(\frac{\partial}{\partial u}\bigg|_{u=0} \frac{p}{2\pi i}\widehat{A}_3(u+v\tau;p\tau)-3v\widehat{A}_3(v\tau;p\tau)\right).
$$
For the case $v=0$, since $A_3(0;\tau)$ is not well-defined, we define
$$
\mathcal{L}_p(0):=\frac{pq^{\frac{-s_p}{p}}}{(q^p;q^p)_\infty}\left(\frac{1}{2\pi i}\frac{\partial}{\partial u}\bigg|_{u=0} \left(A_3(u;p\tau)-\frac{e^{3\pi i u}}{1-e^{2\pi i u}}\right)-\frac{1}{8}E_2(p\tau)-\frac{11}{24}\right),
$$
and
$$
\widehat{\mathcal{L}}_p(0):=\frac{pq^{\frac{-s_p}{p}}}{(q^p;q^p)_\infty}\left(\frac{1}{2\pi i}\frac{\partial}{\partial u}\bigg|_{u=0} \left(\widehat{A}_3(u;p\tau)-\frac{e^{3\pi i u}}{1-e^{2\pi i u}}\right)-\frac{1}{8}E_2(p\tau)-\frac{11}{24}\right).
$$
The function $E_2(\tau)$ is the weight 2 Eisenstein series
$$
E_2(\tau):=1-24\sum_{n=1}^\infty \frac{nq^n}{1-q^n}.
$$
By \eqref{a3m} we have
\begin{align*}
&\frac{1}{2\pi i (q^p;q^p)_\infty}\frac{\partial}{\partial u}\bigg|_{u=0} \left(\widehat{A}_3(u+v\tau;p\tau)-A_3(u+v\tau;p\tau)\right)\\
=&\frac{1}{4\pi (q^p;q^p)_\infty}\frac{\partial}{\partial u}\bigg|_{u=0} \big(e^{2\pi i (u+v\tau)}\theta(p\tau;3p\tau)R(3u+3v\tau-p\tau;3p\tau)\\
&+e^{4\pi i (u+v\tau)}\theta(2p\tau;3p\tau)R(3u+3v\tau-2p\tau;3p\tau)\big)\\
=&-\frac{1}{4}q^{-\frac{1}{2}v}
\sum_{n\in \mathbb{Z}}\left(\sgn\left(n+\frac{1}{2}\right)
-E\left(\left(pn+v+\frac{p}{6}\right)\sqrt{\frac{6y}{p}}\right)\right)(-1)^n(6n+1)q^{-\frac{3pn^2+(6v+p)n}{2}}\\
&-\frac{1}{4}q^{\frac{1}{2}v}
\sum_{n\in \mathbb{Z}}\left(\sgn\left(n+\frac{1}{2}\right)
-E\left(\left(pn+v-\frac{p}{6}\right)\sqrt{\frac{6y}{p}}\right)\right)(-1)^n(6n-1)q^{-\frac{3pn^2+(6v-p)n}{2}}.
\end{align*}
Similarly,
\begin{align*}
&\frac{1}{(q^p;q^p)_\infty}\left(\widehat{A}_3(v\tau;p\tau)-A_3(u+v\tau;p\tau)\right)\\
=&\frac{1}{2}q^{-\frac{1}{2}v}
\sum_{n\in \mathbb{Z}}\left(\sgn\left(n+\frac{1}{2}\right)
-E\left(\left(pn+v+\frac{p}{6}\right)\sqrt{\frac{6y}{p}}\right)\right)(-1)^nq^{-\frac{3pn^2+(6v+p)n}{2}}\\
&+\frac{1}{2}q^{\frac{1}{2}v}
\sum_{n\in \mathbb{Z}}\left(\sgn\left(n+\frac{1}{2}\right)
-E\left(\left(pn+v-\frac{p}{6}\right)\sqrt{\frac{6y}{p}}\right)\right)(-1)^nq^{-\frac{3pn^2+(6v-p)n}{2}}.
\end{align*}
Then we have
\begin{align*}
&\widehat{\mathcal{L}}_p(v)-\mathcal{L}_p(v)\\
=&-\frac{(-1)^v}{4}q^{\frac{-3v^2-pv-2s_p}{2p}}
\sum_{n\in \mathbb{Z}}\left(\sgn\left(n+\frac{1}{2}\right)
-E\left(\left(pn+v+\frac{p}{6}\right)\sqrt{\frac{6y}{p}}\right)\right)\\
&\times(-1)^n(6np+6v+p)q^{-\frac{3pn^2+(6v+p)n}{2}}\\
&-\frac{(-1)^v}{4}q^{\frac{-3v^2+pv-2s_p}{2p}}
\sum_{n\in \mathbb{Z}}\left(\sgn\left(n+\frac{1}{2}\right)
-E\left(\left(pn+v-\frac{p}{6}\right)\sqrt{\frac{6y}{p}}\right)\right)\\
&\times(-1)^n(6np+6v-p)q^{-\frac{3pn^2+(6v-p)n}{2}}.
\end{align*}
If $p\equiv 1 \pmod{6}$, then we let $m_1=\frac{p-1}{6}-v$ and $m_2=\frac{p-1}{6}+v$. We have
\beq
\label{ulth1}
\widehat{\mathcal{L}}_p(v)-\mathcal{L}_p(v)=(-1)^{\frac{p-1}{6}}\frac{1}{4}R_{p,n}+\epsilon^{*}_p(v),
\eeq
where $2n+6v+1\equiv 0 \pmod{p}$, $R_{p,n}$ was defined in Lemma \ref{lmuf} and $\epsilon^{*}_p(v)=(-1)^v\frac{p-6v}{2}q^{\frac{v(p-3v)-2s_p}{2p}}$ if $0\leq v<\frac{p-1}{6}$, $\epsilon^{*}_p(v)=(-1)^v\frac{5p-6v}{2}q^{\frac{(p-v)(3v-2p)-2s_p}{2p}}$ if $\frac{5p+1}{6}\leq v<1$ and $\epsilon^{*}_p(v)=0$ if $\frac{p+1}{6}\leq v<\frac{5p-1}{6}$. If $p\equiv -1 \pmod{6}$, then we let $m_1=-\frac{p+1}{6}-v$ and $m_2=-\frac{p+1}{6}+v$. Similarly we have
\beq
\label{ulth2}
\widehat{\mathcal{L}}_p(v)-\mathcal{L}_p(v)=(-1)^{\frac{p+1}{6}}\frac{1}{4}R_{p,n}+\epsilon^{\#}_p(v),
\eeq
where $\epsilon^{\#}_p(v)=(-1)^v\frac{p-6v}{2}q^{\frac{v(p-3v)-2s_p}{2p}}$ if $0\leq v<\frac{p+1}{6}$, $\epsilon^{\#}_p(v)=(-1)^v\frac{5p-6v}{2}q^{\frac{(p-v)(3v-2p)-2s_p}{2p}}$ if $\frac{5p-1}{6}\leq v<1$ and $\epsilon^{\#}_p(v)=0$ if $\frac{p-1}{6}\leq v<\frac{5p+1}{6}$.
By \eqref{ulth1}, \eqref{ulth2} and Lemma \ref{lmuf}, we arrived at the following theorem.

\begin{theorem}
\label{thm:UFminusL}
Let $0\leq v_n<p$ such that $2n+6v_n+1\equiv 0\pmod{p}$.

(1)If $6k+s^2-s\equiv 0 \pmod{p}$ then
$$
U_{p,k}(\widehat{\mathcal{F}}_{p,s}(\tau))-\chi_{12}(p)\widehat{\mathcal{L}}_p(v_{s-1})
=U_{p,k}(\mathcal{F}_{p,s}(\tau))-\chi_{12}(p)(\mathcal{L}_p(v_{s-1})+\epsilon_p(v_{s-1})),
$$

(2)If $6k+s^2+s\equiv 0 \pmod{p}$ then
$$
U_{p,k}(\widehat{\mathcal{F}}_{p,s}(\tau))+\chi_{12}(p)\widehat{\mathcal{L}}_p(v_s)
=U_{p,k}(\mathcal{F}_{p,s}(\tau))+\chi_{12}(p)(\mathcal{L}_p(v_s)+\epsilon_p(v_s)),
$$

(3)If $6k+s^2\pm s \not\equiv 0 \pmod{p}$ then
$$
U_{p,k}(\widehat{\mathcal{F}}_{p,s}(\tau))=U_{p,k}(\mathcal{F}_{p,s}(\tau)),
$$
where $\epsilon_p(v)=(-1)^v\frac{p-6v}{2}q^{\frac{v(p-3v)-2s_p}{2p}}$ if $0\leq v<\frac{p}{6}$, $\epsilon_p(v)=(-1)^v\frac{5p-6v}{2}q^{\frac{(p-v)(3v-2p)-2s_p}{2p}}$ if $\frac{5p}{6}<v<1$ and $\epsilon_p(v)=0$ if $\frac{p}{6}<v<\frac{5p}{6}$.
\end{theorem}

\subsection{The main theorems}

\begin{lemma}
\label{lemma:L0Fmodular}
Let
$$
F(\tau)=\frac{1}{2\pi i}\frac{\partial}{\partial u}\bigg|_{u=0} \left(\widehat{A}_3(u;\tau)-\frac{e^{3\pi i u}}{1-e^{2\pi i u}}\right)-\frac{1}{8}E_2(\tau)-\frac{11}{24}.
$$
Then for $\abcdMAT \in \mathrm{SL}_2(\mathbb{Z})$, we have
\beq
\label{lmF}
F\left(\frac{a\tau+b}{c\tau+d}\right)=(c\tau+d)^2F(\tau).
\eeq
\end{lemma}

\begin{proof}
Let
$$
f(u,\tau)=\widehat{A}_3(u;\tau)-\frac{e^{3\pi i u}}{1-e^{2\pi i u}}.
$$
By \eqref{Alt} we have
\begin{align}
\label{fut}
f\left(\frac{u}{c\tau+d},\frac{a\tau+b}{c\tau+d}\right)=&\widehat{A}_3\left(\frac{u}{c\tau+d};\frac{a\tau+b}{c\tau+d}\right)-\frac{e^{\frac{3\pi i u}{c\tau+d}}}{1-e^{\frac{2\pi i u}{c\tau+d}}}\\
\nonumber
=&(c\tau+d)e^{\frac{-3\pi i c u^2}{c\tau+d}}\widehat{A}_3(u;\tau)-\frac{e^{\frac{3\pi i u}{c\tau+d}}}{1-e^{\frac{2\pi i u}{c\tau+d}}}\\
\nonumber
=&(c\tau+d)e^{\frac{-3\pi i c u^2}{c\tau+d}}f(u,\tau)+(c\tau+d)e^{\frac{-3\pi i c u^2}{c\tau+d}}\frac{e^{3\pi i u}}{1-e^{2\pi i u}}-\frac{e^{\frac{3\pi i u}{c\tau+d}}}{1-e^{\frac{2\pi i u}{c\tau+d}}}.
\end{align}
Let
$$
g(u,\tau)=(c\tau+d)e^{\frac{-3\pi i c u^2}{c\tau+d}}\frac{e^{3\pi i u}}{1-e^{2\pi i u}}-\frac{e^{\frac{3\pi i u}{c\tau+d}}}{1-e^{\frac{2\pi i u}{c\tau+d}}}.
$$
We calculate that
$$
\frac{1}{2\pi i}\frac{\partial}{\partial u}\bigg|_{u\rightarrow 0}g(u,\tau)=-\frac{3ic}{4\pi}+\frac{11}{24}\left((c\tau+d)^{-1}-(c\tau+d)\right).
$$
Applying $\frac{c\tau+d}{2\pi i}\frac{\partial}{\partial u}\bigg|_{u\rightarrow 0}$ on both sides of \eqref{fut}, we have
\begin{align}
\label{ft}
&\frac{1}{2\pi i}\frac{\partial}{\partial u}\bigg|_{u=0}f\left(u,\frac{a\tau+b}{c\tau+d}\right)\\
\nonumber
=&(c\tau+d)^2\frac{1}{2\pi i}\frac{\partial}{\partial u}\bigg|_{u=0}
f(u,\tau)-\frac{3ic}{4\pi}(c\tau+d)+\frac{11}{24}\left(1-(c\tau+d)^2\right).
\end{align}
Let
$$
h(\tau)=\frac{3}{8\pi y}+\frac{11}{24},
$$
with $\tau=x+yi$. Then
\beq
h\left(\frac{a\tau+b}{c\tau+d}\right)=(c\tau+d)^2h(\tau)-\frac{3ic}{4\pi}(c\tau+d)+\frac{11}{24}\left(1-(c\tau+d)^2\right).
\eeq
Let
$$
\widehat{E}_2(\tau)=-\frac{3}{\pi y}+E_{2}(\tau).
$$
It is well-known that
\beq
\label{e2t}
\widehat{E}_2\left(\frac{a\tau+b}{c\tau+d}\right)=(c\tau+d)^2\widehat{E}_2(\tau).
\eeq
Combining \eqref{ft}-\eqref{e2t} we have \eqref{lmF}.
\end{proof}

Now we give the proof of the main theorem. Let $D(a,M,k)$ be the $M$-dissection of $D(a,M)$, namely
$$
D(a,M,k):=\sum_{n=0}^\infty \left(N(a,M,Mn+k)-\frac{p(Mn+k)}{M}\right)q^n,
$$
and recall the definition of $L_p(v)$ \eqref{lpp} and \eqref{lp0}
\begin{theorem}
\label{mainN}
Let $p\geq 5$ be a prime and $0<s<\frac{p}{2}$, $0\leq k < p$. For each integer $m$ let $v_m$ be the integer such that $0\leq v_m<p$ and $2m+6v_m+1\equiv 0\pmod{p}$. Let $c(k,v)=\frac{3v(p-v)}{2p}-\frac{k+s_p}{p}$.

(1)If $6k+s^2-s\equiv 0 \pmod{p}$ then
\begin{align*}
&\sum_{n\geq n_0}\left(NT(s,p,pn+k)-NT(p-s,p,pn+k)\right)q^n\\
=&\sum_{r=1}^{p-1}\frac{p-2r}{2p}D(r-s,p,k)+\chi_{12}(p)q^{c(k,v_{s-1})}L_p(v_{s-1})+N_p(s,k).
\end{align*}

(2)If $6k+s^2+s\equiv 0 \pmod{p}$ then
\begin{align*}
&\sum_{n\geq n_0}\left(NT(s,p,pn+k)-NT(p-s,p,pn+k)\right)q^n\\
=&\sum_{r=1}^{p-1}\frac{p-2r}{2p}D(r-s,p,k)-\chi_{12}(p)q^{c(k,v_{s})}L_p(v_{s})+N_p(s,k).
\end{align*}

(3)If $6k+s^2\pm s \not\equiv 0 \pmod{p}$ then
\begin{align*}
&\sum_{n\geq n_0}\left(NT(s,p,pn+k)-NT(p-s,p,pn+k)\right)q^n\\
=&\sum_{r=1}^{p-1}\frac{p-2r}{2p}D(r-s,p,k)+N_p(s,k).
\end{align*}

The part $N_p(s,k)$ has the property that
\begin{equation}
\label{eq:Nmodular}
\frac{(q,q^{p-1};q^p)_\infty^2}{q(q^p;q^p)_\infty^3}G_p^{k+s_p+1}N_p(s,k)
\end{equation}
is a modular function on $\Gamma_1(p)$ (with trivial multiplier system), where
$$
G_p=\frac{(q^{(p-1)/2},q^{(p+1)/2};q^p)_\infty}{(q^{(p-3)/2},q^{(p+3)/2};q^p)_\infty}.
$$
\end{theorem}

Before giving the proof, we recall the definition of modular functions: we say that $f$ is a modular function on $\Gamma_1(p)$ (with trivial multiplier system) if $f$ is meromorphic on $\uhp$ and $f\left(\frac{a\tau+b}{c\tau+d}\right)=f(\tau)$ for any $\tbtmat{a}{b}{c}{d}\in\Gamma_1(p)$ and $\ord_{a/c}f>-\infty$ for any coprime integers $a$ and $c$.

\begin{proof}
First we consider the case $6k+s^2\pm s \not\equiv 0 \pmod{p}$. Applying the operator $q^{-\frac{k}{p}}U_{p,k}$ to both sides of \eqref{maoidp} we find that $N_p(s,k)=q^{-k/p}U_{p,k}\mathcal{F}_{p,s}$. Therefore, by Theorem \ref{thm:UFminusL}(3), $N_p(s,k)=q^{-k/p}U_{p,k}\widehat{\mathcal{F}}_{p,s}$ which is holomorphic on $\uhp$. Taking into account of \eqref{eq:defHatF}, we have
\begin{align*}
N_p(s,k)=\frac{q^{-k/p}}{2p\pi i}\sum_{j=1}^{p-1}\zeta_{p}^{j(s-1/2)}(1-\zeta_{p}^j)\cdot U_{p,k}\left(\frac{\partial\widehat{A}_3\left(-j/p\right)}{(q)_\infty}\right).
\end{align*}
Thus, to prove the modularity of \eqref{eq:Nmodular} it suffices to prove the modularity of
\begin{equation}
\label{eq:UpkFpsModular}
\frac{(q,q^{p-1};q^p)_\infty^2}{q(q^p;q^p)_\infty^3}\cdot\frac{(q^{(p-1)/2},q^{(p+1)/2};q^p)_\infty^{k+s_p+1}}{(q^{(p-3)/2},q^{(p+3)/2};q^p)_\infty^{k+s_p+1}}q^{-\frac{k}{p}}U_{p,k}\left(\frac{\partial\widehat{A}_3\left(-j/p\right)}{(q;q)_\infty}\right)
\end{equation}
for each $0<j<p$. When $p\geq7$, the above function is the function $f$ in Theorem \ref{thm:AppellGEtaModular} with $\ell=3$,
\begin{equation*}
r_1=2,\quad r_{(p-1)/2}=k+\frac{p^2-1}{24}+1,\quad r_{(p-3)/2}=-\left(k+\frac{p^2-1}{24}+1\right)
\end{equation*}
and $r_j=0$ for other $j$. Since in this case
\begin{align*}
12\cdot\sum_{\delta=1}^{(p-1)/2}r_\delta\cdot\delta^2&=24 +12\cdot\left(k+\frac{p^2-1}{24}+1\right)\cdot\left(\left(\frac{p-1}{2}\right)^2-\left(\frac{p-3}{2}\right)^2\right)\\
&\equiv-24k+1\pmod{p},
\end{align*}
\eqref{eq:UpkFpsModular} and hence \eqref{eq:Nmodular} satisfy the modular transformation equations on $\Gamma_1(p)$ with trivial multiplier system by Theorem \ref{thm:AppellGEtaModular}. It follows from this, the fact $N_p(s,k)$ is holomorphic on $\uhp$, Lemma \ref{lemma:ordUpkprime} and \eqref{eq:ordGEta} that \eqref{eq:Nmodular} is a modular function on $\Gamma_1(p)$. The case $p=5$ is proved similarly with the only change
$$
r_1=r_{(p-3)/2}=2-\left(k+\frac{p^2-1}{24}+1\right).
$$

Next we consider the case $6k+s^2-s\equiv 0 \pmod{p}$. Applying the operator $q^{-\frac{k}{p}}U_{p,k}$ to both sides of \eqref{maoidp} and using Theorem \ref{thm:UFminusL}(1) we find that
\begin{align*}
N_p(s,k)&=q^{-k/p}U_{p,k}\mathcal{F}_{p,s}-\chi_{12}(p)q^{c(k,v_{s-1})}L_p(v_{s-1})\\
&=q^{-k/p}\cdot\left(U_{p,k}\widehat{\mathcal{F}}_{p,s}-\chi_{12}(p)\widehat{\mathcal{L}}_p(v_{s-1})\right),
\end{align*}
which is holomorphic on $\uhp$. Note that for $v_{s-1}\not\equiv0\pmod{p}$,
$$
\widehat{\mathcal{L}}_p(v_{s-1})=\frac{(-1)^{v_{s-1}}p}{2\pi i}\cdot q^{\frac{1}{24p}}\eta(p\tau)^{-1}g_{3,v_{s-1}/p}(p\tau)
$$
where $g_{3,v_{s-1}/p}$ is defined in \eqref{eq:defglx}. It follows that
\begin{multline*}
N_p(s,k)=\frac{q^{-k/p}}{2p\pi i}\sum_{j=1}^{p-1}\zeta_{p}^{j(s-1/2)}(1-\zeta_{p}^j)\cdot U_{p,k}\left(\frac{\partial\widehat{A}_3\left(-j/p\right)}{(q;q)_\infty}\right)\\
-\chi_{12}(p)\frac{(-1)^{v_{s-1}}p}{2\pi i}\cdot q^{\frac{1}{24p}-\frac{k}{p}}\eta(p\tau)^{-1}g_{3,v_{s-1}/p}(p\tau),
\end{multline*}
Since $6v_{s-1}+2s-1\equiv0\pmod{p}$ and $6k+s^2-s\equiv 0 \pmod{p}$, we have $36v_{s-1}^2\equiv-24k+1\pmod{p}$. Hence, according to Lemma \ref{lemm:UpkprimeTransformation}, $q^{-\frac{1}{24p}}U_{p,k}\left(\frac{\partial\widehat{A}_3\left(-j/p\right)}{(q;q)_\infty}\right)$ and $\eta(p\tau)^{-1}g_{3,v_{s-1}/p}(p\tau)$ satisfy the same modular transformation equations on $\Gamma_1(p)$ (i.e., they have the same multiplier system). Therefore, proving the modularity of \eqref{eq:Nmodular}, as in the last case, is equivalent to proving the modularity of \eqref{eq:UpkFpsModular} which can be done in the same manner. Note that when proving the holomorphicity at cusps, we need Corollaries \ref{coro:Hpart} and \ref{cor:glxOrder} in addition. This concludes the proof of the case $v_{s-1}\not\equiv0\pmod{p}$. Otherwise, let us consider the case $v_{s-1}=0$ in which situation $0=36v_{s-1}^2\equiv-24k+1\pmod{p}$. By the last assertion of Lemma \ref{lemm:UpkprimeTransformation} and Lemma \ref{lemma:L0Fmodular} proving the modularity of \eqref{eq:Nmodular} is still equivalent to proving the modularity of \eqref{eq:UpkFpsModular} which can be done as above.

Finally, for the case $6k+s^2+s\equiv 0 \pmod{p}$, the proof is almost the same as the last case $6k+s^2-s\equiv 0 \pmod{p}$. We omit the redundant details.
\end{proof}

Let
$$
D_C(a,M):=\sum_{n=0}^\infty\left(M(a,M,n)-\frac{p(n)}{M}\right)q^n,
$$
and
$$
D_C(a,M,k):=\sum_{n=0}^\infty \left(M(a,M,Mn+k)-\frac{p(Mn+k)}{M}\right)q^n.
$$
From \eqref{maoidm} we deduce the following theorem which is the $M_\omega$-analogue of Theorem~\ref{mainN}.
\begin{theorem}
\label{mainM}
Under the same condition and notation of Theorem \ref{mainN},
\begin{align*}
&\sum_{n\geq n_0}\left(M_\omega(s,p,pn+k)-M_\omega(p-s,p,pn+k)\right)q^n\\
=&\sum_{r=1}^{p-1}\frac{p-2r}{2p}D_C(r-s,p,k)+M_p(s,k).
\end{align*}
The part $M_p(s,k)$ has the property that
\begin{equation}
\label{eq:MpskModular}
\frac{(q,q^{p-1};q^p)_\infty^2}{q(q^p;q^p)_\infty^3}G_p^{k+s_p+1}M_p(s,k)
\end{equation}
is a modular function on $\Gamma_1(p)$.
\end{theorem}
The proof proceeds as that of Theorem \ref{mainN}, with the roles $A_3$ and $\widehat{A}_3$ replaced by $A_1$ and $\widehat{A}_1$ (according to \eqref{maoidm}). Since $A_1$ is holomorphic, that is, $A_1=\widehat{A}_1$, no nonholomorphic correction term is needed. Thus, to prove the modularity of \eqref{eq:MpskModular} it suffices to prove the modularity of \eqref{eq:UpkFpsModular} with $\widehat{A}_3$ replaced by $\widehat{A}_1$ which follows from Theorem \ref{thm:AppellGEtaModular} (with $\ell=1$) as well.

\section{Examples and representations of $N_p(s,k)$ and $M_p(s,k)$}
\label{sec:Examples and the algorithm}

\subsection{Examples}

In this section, we show how Theorems~\ref{mainN} and~\ref{mainM} work on proving identities such as \eqref{ex1}-\eqref{ex3}. The part $D(s,p,k)$ was obtained by Hickerson and Mortenson \cite{HiMo-17}. Denote
$$
J_k:=(q^k;q^k)_\infty,
$$
$$
J_{k,a}:=(q^a,q^{k-a},q^k;q^k)_\infty,
$$
and
$$
g(x,q):=x^{-1}\left(-1+\sum_{n=0}^\infty \frac{q^{n^2}}{(x;q)_{n+1}(q/x;q)_n}\right)
$$
be a universal mock theta function. They state that for $M=5$ \cite[Eq(12)-Eq(14)]{HiMo-17}
\begin{align}
\label{D50}
D(0,5)=&-2q^5g(q^5,q^{25})+\frac{4}{5}\cdot\frac{J_5J_{25}^3}{J_{25,5}^3}\\
\nonumber
&+\frac{4q}{5}\cdot \frac{J_{25}^2}{J_{25,5}}-\frac{2q^2}{5}\cdot \frac{J_{25}^2}{J_{25,10}}+\frac{2q^3}{5}\cdot \frac{J_5J_{25}^3}{J_{25,10}^3},\\
\label{D51}
D(1,5)=&D(4,5)=q^5g(q^5,q^{25})-q^8g(q^{10},q^{25})-\frac{1}{5}\cdot\frac{J_5J_{25}^3}{J_{25,5}^3}\\
\nonumber
&-\frac{q}{5}\cdot \frac{J_{25}^2}{J_{25,5}}+\frac{3q^2}{5}\cdot \frac{J_{25}^2}{J_{25,10}}-\frac{3q^3}{5}\cdot \frac{J_5J_{25}^3}{J_{25,10}^3},\\
\label{D52}
D(2,5)=&D(3,5)=q^8g(q^{10},q^{25})-\frac{1}{5}\cdot\frac{J_5J_{25}^3}{J_{25,5}^3}\\
\nonumber
&-\frac{q}{5}\cdot \frac{J_{25}^2}{J_{25,5}}-\frac{2q^2}{5}\cdot \frac{J_{25}^2}{J_{25,10}}+\frac{2q^3}{5}\cdot \frac{J_5J_{25}^3}{J_{25,10}^3},
\end{align}
and for $M=7$ \cite[Eq(34)-Eq(37)]{HiMo-17}
\begin{align}
\label{D70}
D(0,7)=&2+2q^7g(q^7,q^{49})\\
\nonumber
&-\frac{8}{7}\cdot\frac{J_{49,21}^2}{J_7}+\frac{6q}{7}\cdot \frac{J_{49}^2}{J_{49,7}}-\frac{2q^2}{7}\cdot \frac{J_{49,14}^2}{J_7}\\
\nonumber
&+\frac{4q^3}{7}\cdot \frac{J_{49}^2}{J_{49,14}}+\frac{2q^4}{7}\cdot \frac{J_{49}^2}{J_{49,21}}
-\frac{4q^6}{7}\cdot \frac{J_{49,7}^2}{J_7},\\
\label{D71}
D(1,7)=&D(6,7)=-1-q^7g(q^7,q^{49})+q^{16}g(q^{21},q^{49})\\
\nonumber
&+\frac{6}{7}\cdot\frac{J_{49,21}^2}{J_7}-\frac{q}{7}\cdot \frac{J_{49}^2}{J_{49,7}}+\frac{5q^2}{7}\cdot \frac{J_{49,14}^2}{J_7}\\
\nonumber
&-\frac{3q^3}{7}\cdot \frac{J_{49}^2}{J_{49,14}}+\frac{2q^4}{7}\cdot \frac{J_{49}^2}{J_{49,21}}
+\frac{3q^6}{7}\cdot \frac{J_{49,7}^2}{J_7},\\
\label{D72}
D(2,7)=&D(5,7)=q^{13}g(q^{14},q^{49})-q^{16}g(q^{21},q^{49})\\
\nonumber
&-\frac{1}{7}\cdot\frac{J_{49,21}^2}{J_7}-\frac{q}{7}\cdot \frac{J_{49}^2}{J_{49,7}}-\frac{2q^2}{7}\cdot \frac{J_{49,14}^2}{J_7}\\
\nonumber
&+\frac{4q^3}{7}\cdot \frac{J_{49}^2}{J_{49,14}}-\frac{5q^4}{7}\cdot \frac{J_{49}^2}{J_{49,21}}
+\frac{3q^6}{7}\cdot \frac{J_{49,7}^2}{J_7},\\
\label{D73}
D(3,7)=&D(4,7)=-q^{13}g(q^{14},q^{49})\\
\nonumber
&-\frac{1}{7}\cdot\frac{J_{49,21}^2}{J_7}-\frac{q}{7}\cdot \frac{J_{49}^2}{J_{49,7}}-\frac{2q^2}{7}\cdot \frac{J_{49,14}^2}{J_7}\\
\nonumber
&-\frac{3q^3}{7}\cdot \frac{J_{49}^2}{J_{49,14}}+\frac{2q^4}{7}\cdot \frac{J_{49}^2}{J_{49,21}}
-\frac{4q^6}{7}\cdot \frac{J_{49,7}^2}{J_7}.
\end{align}
Analogously, Mortenson \cite[Theorem 10.1]{Mo-19} calculated the generating functions of $D_C(a,M,k)$ by
\begin{align}
\label{Dc50}
D_C(0,5)=&\frac{4}{5}\cdot\frac{J_5J_{25}^3}{J_{25,5}^3}-\frac{6q}{5}\cdot \frac{J_{25}^2}{J_{25,5}}\\
\nonumber
&-\frac{2q^2}{5}\cdot \frac{J_{25}^2}{J_{25,10}}+\frac{2q^3}{5}\cdot \frac{J_5J_{25}^3}{J_{25,10}^3},\\
\label{Dc51}
D_C(1,5)=&D_C(4,5)=-\frac{1}{5}\cdot\frac{J_5J_{25}^3}{J_{25,5}^3}+\frac{4q}{5}\cdot \frac{J_{25}^2}{J_{25,5}}\\
\nonumber
&-\frac{2q^2}{5}\cdot \frac{J_{25}^2}{J_{25,10}}-\frac{3q^3}{5}\cdot \frac{J_5J_{25}^3}{J_{25,10}^3},\\
\label{Dc52}
D_C(2,5)=&D_C(3,5)=-\frac{1}{5}\cdot\frac{J_5J_{25}^3}{J_{25,5}^3}-\frac{q}{5}\cdot \frac{J_{25}^2}{J_{25,5}}\\
\nonumber
&\frac{3q^2}{5}\cdot \frac{J_{25}^2}{J_{25,10}}+\frac{2q^3}{5}\cdot \frac{J_5J_{25}^3}{J_{25,10}^3},\\
\end{align}
and \cite[Theorem 10.2]{Mo-19}
\begin{align}
\label{Dc70}
D_C(0,7)=&\frac{6}{7}\cdot\frac{J_{49,21}^2}{J_7}-\frac{8q}{7}\cdot \frac{J_{49}^2}{J_{49,7}}-\frac{2q^2}{7}\cdot \frac{J_{49,14}^2}{J_7}\\
\nonumber
&+\frac{4q^3}{7}\cdot \frac{J_{49}^2}{J_{49,14}}+\frac{2q^4}{7}\cdot \frac{J_{49}^2}{J_{49,21}}
-\frac{4q^6}{7}\cdot \frac{J_{49,7}^2}{J_7},\\
\label{Dc71}
D_C(1,7)=&D_C(6,7)=-\frac{1}{7}\cdot\frac{J_{49,21}^2}{J_7}+\frac{6q}{7}\cdot \frac{J_{49}^2}{J_{49,7}}-\frac{2q^2}{7}\cdot \frac{J_{49,14}^2}{J_7}\\
\nonumber
&-\frac{3q^3}{7}\cdot \frac{J_{49}^2}{J_{49,14}}-\frac{5q^4}{7}\cdot \frac{J_{49}^2}{J_{49,21}}
+3\frac{q^6}{7}\cdot \frac{J_{49,7}^2}{J_7},\\
\label{Dc72}
D_C(2,7)=&D_C(5,7)=-\frac{1}{7}\cdot\frac{J_{49,21}^2}{J_7}-\frac{q}{7}\cdot \frac{J_{49}^2}{J_{49,7}}+\frac{5q^2}{7}\cdot \frac{J_{49,14}^2}{J_7}\\
\nonumber
&-\frac{3q^3}{7}\cdot \frac{J_{49}^2}{J_{49,14}}+\frac{2q^4}{7}\cdot \frac{J_{49}^2}{J_{49,21}}
-\frac{4q^6}{7}\cdot \frac{J_{49,7}^2}{J_7},\\
\label{Dc73}
D_C(3,7)=&D_C(4,7)=-\frac{1}{7}\cdot\frac{J_{49,21}^2}{J_7}-\frac{q}{7}\cdot \frac{J_{49}^2}{J_{49,7}}-\frac{2q^2}{7}\cdot \frac{J_{49,14}^2}{J_7}\\
\nonumber
&\frac{4q^3}{7}\cdot \frac{J_{49}^2}{J_{49,14}}+\frac{2q^4}{7}\cdot \frac{J_{49}^2}{J_{49,21}}
+\frac{3q^6}{7}\cdot \frac{J_{49,7}^2}{J_7}.\\
\end{align}

By the representations of $N_p(s,k)$ and $M_p(s,k)$ for $p=5$ and $p=7$ in the appendices, one can obtain \eqref{ex2} from Theorem~\ref{mainN}, Theorem~\ref{mainM}, \eqref{D50}-\eqref{D52} and \eqref{Dc50}-\eqref{Dc52} immediately. Similarly we have
\begin{align*}
&\sum_{n=0}^\infty \big(NT(1,7,7n+5)-NT(6,7,7n+5)+3NT(2,7,7n+5)-3NT(5,7,7n+5\big)q^n\\
\nonumber
=&-7\frac{J_7^5J_{7,2}^{10}}{J_{7,1}^5J_{7,3}^7}(1-t)^4,
\end{align*}
where $t=q\frac{J_{7,1}^3}{J_{7,2}^2J_{7,3}}$. We use Frye and Garvan's MAPLE package \cite{gamp} to verify
$$
1-t=\frac{J_{7,1}J_{7,3}^2}{J_{7,2}^3}
$$
computationally and then arrive at \eqref{ex1}. We do not list the case $M_{11}(s,k)$ but we mention that the relations, which have been proved in \cite[Appendix]{ccy}, can also be demonstrated by our method. It follows directly that the equation
$$
\sum_{m=1}^{(p-1)/2}m\left[M_{\omega}(m,p,pn-s_p)-M_{\omega}(p-m,p,pn-s_p)\right]=0,
$$
holds for $p=5,7$, and $11$.

\subsection{Behavior of $N_p(s,k)$ and $M_p(s,k)$ at cusps}
\label{subsec:Behaviour at cusps}

In this subsection, we derive the representations of $N_p(s,k)$ and $M_p(s,k)$ presented in Appendices \ref{apx:A} and \ref{apx:B}. Similar representations of $M_{11}(s,k)$ were previously obtained in \cite{ccy}, where the authors employ complicated $q$-series techniques in their proof. Here we prove identities involving $N_p(s,k)$ or $M_p(s,k)$ using the valence formula. We need to know up to which power of $q$ we must check that the Fourier expansions of \eqref{eq:Nmodular} and of some modular function are the same.

Recall the cusps of $\Gamma_1(p)\backslash\uhp$ can be represented by the set
\begin{align*}
\mathscr{R}_p:=&\left\{\frac{f(a_0,c)}{c}\colon 1\leq c< \frac{p}{2},\quad0\leq a_0<(c,p),\quad(a_0,c,p)=1\right\}\\
\cup&\left\{\frac{f(a_0,c)}{c}\colon c=\frac{p}{2} \text{ or }c=p,\quad0\leq a_0\leq\frac{(c,p)}{2},\quad(a_0,c,p)=1\right\}
\end{align*}
where $c,\,a_0$ take values in $\Z$ and $f(a_0,c)$ is any integer such that $(f(a_0,c),c)=1$ and $f(a_0,c)\equiv a_0\pmod{p}$; c.f. \cite[Corollary 6.3.19]{CoSt-17}. Let $a_1/c_1$ and $a_2/c_2$ be two points in the projective line over $\Q$ where $a_1,c_1,a_2,c_2$ are integers with $(a_1,c_1)=(a_2,c_2)=1$. Then they represent the same cusp of $\Gamma_1(p)\backslash\uhp$, that is, there exists $\gamma\in\Gamma_1(p)$ such that $\gamma\frac{a_1}{c_1}=\frac{a_2}{c_2}$, if and only if there is an $\varepsilon\in\{\pm1\}$ such that $c_1\equiv\varepsilon c_2\pmod{p}$ and $a_1\equiv\varepsilon a_2\pmod{(p,c_1)}$. Moreover, if $p\neq4$, then the width of the cusp representative $\frac{a}{c}$ equals $\frac{p}{(c,p)}$ where $a$, $c$ are coprime integers. In particular, if $p$ is an odd prime, then
\begin{equation*}
\mathscr{R}_p=\left\{\frac{p}{c}\colon 1\leq c\leq \frac{p-1}{2}\right\}\cup\left\{\frac{a}{p}\colon 1\leq a\leq \frac{p-1}{2}\right\}.
\end{equation*}
The width of any cusp in the former set of the right-hand side above equals $p$ and that in the latter set equals $1$. Note that the cusp $i\infty$ is represented by $\frac{1}{p}$.

For a meromorphic function $f$ on $\uhp$ that satisfies the modular transformation equations of $\Gamma_1(p)$ ($p\neq4$) of an integral or half-integral weight (possibly with a multiplier system) we set $\mathrm{div}_{a/c}f=\frac{p}{(c,p)}\ord_{a/c}f$ (the values $-\infty$ and $+\infty$, which represent that $f$ has an essential singularity at $a/c$ and $f=0$ respectively, are allowed). For $\tau_0\in\uhp$, if $f(\tau)=c\cdot(\tau-\tau_0)^n(1+o(1))$ as $\tau\to\tau_0$, set $\mathrm{div}_{\tau_0}f=\frac{n}{e_{\tau_0}}$ where
$$
e_{\tau_0}:=\#\{\gamma\in\Gamma_1(p)/\{\pm I\}\colon \gamma\tau_0=\tau_0\}.
$$
It is known that when $p>3$ we have $e_{\tau_0}=1$ for any $\tau_0\in\uhp$ (c.f. \cite[Ex 2.3.7, p.57]{DiSh-05}). Clearly, $\mathrm{div}_{x}(fg)=\mathrm{div}_{x}(f)+\mathrm{div}_{x}(g)$ where $x\in\uhp\cup\mathscr{R}_p$ and $g$ satisfies the same conditions for $f$ but with possibly different weight and multiplier system, while $\mathrm{div}_{x}(f+g)\geq\min\{\mathrm{div}_{x}(f),\mathrm{div}_{x}(g)\}$ where the weight and multiplier system of $f$ and $g$ are required to be the same. (A sufficient condition for $\mathrm{div}_{x}(f+g)=\min\{\mathrm{div}_{x}(f),\mathrm{div}_{x}(g)\}$ is $\mathrm{div}_{x}(f)\neq\mathrm{div}_{x}(g)$.) A special case of the well known valence formula states that, if $f$ is nonzero, of weight $0$ and meromorphic at all cusps, then
\begin{equation}
\label{eq:valence}
\sum_{a/c\in\mathscr{R}_p}\mathrm{div}_{a/c}f+\sum_{\tau\in\Gamma_1(p)\backslash\uhp}\mathrm{div}_{\tau}f=0.
\end{equation}
For a proof of the valence formula that allows half-integral weights and any multiplier systems, see \cite[Theorem 2.1 and Eq. (13)]{ZZ23}.


More generally, if $f$ is only real analytic on $\uhp$ instead of meromorphic and has well defined holomorphic part $H_{1}$ and nonholomorphic part $H_{2}$ (c.f. Remark \ref{rema:holoPart}), then for $\tau\in\uhp$, define $\mathrm{div}_{\tau}f=\mathrm{div}_{\tau}H_{1}$. For the cusp representative $a/c\in\mathscr{R}_p$, let $\gamma\in\slZ$ satisfy $\gamma(i\infty)=a/c$. If $f\vert_r\gamma$ ($r$ is the weight) has well defined holomorphic part $H_{\gamma,1}$ and nonholomorphic part $H_{\gamma,2}$, then define $\mathrm{div}_{a/c}f=\frac{p}{(c,p)}\mathrm{ord}_{i\infty}H_{\gamma,1}$ ($p\neq 4$). Clearly, this is independent of the choice of $\gamma$.

\begin{prop}
\label{prop:divNmodular}
Let the notation be as in Theorem \ref{mainN} and let $f$ denote the function \eqref{eq:Nmodular}. Then
\begin{align}
\mathrm{div}_{\frac{p}{c}}f&> -s_p,\quad&\text{for } &1\leq c\leq \frac{p-1}{2},\label{eq:divfpc}\\
\mathrm{div}_{\frac{a}{p}}f&> -\frac{(k+s_p+3)p}{8},\quad&\text{for } &1\leq a\leq \frac{p-1}{2},\label{eq:divfap}\\
\mathrm{div}_{\tau}f&\geq0,\quad&\text{for } &\tau\in\uhp\label{eq:divftau}.
\end{align}
\end{prop}
\begin{proof}
Set
\begin{align*}
f_j(\tau)&=\frac{\eta(p\tau)}{\eta(\tau)}U'_{p,k}\left(\frac{\eta(p\tau)}{\eta(\tau)}\partial\widehat{A}_3\left(-\frac{j}{p}\right)\right),\quad j=1,2,\dots,p-1,\\
g_v(\tau)&=g_{3,v/p}(p\tau),\quad v=1,2,\dots,p-1,\\
g_0(\tau)&=F(p\tau),\\
h(\tau)&=\eta_{p,1}(\tau)^2\eta_{p,(p-1)/2}(\tau)^{k+s_p+1}\eta_{p,(p-3)/2}(\tau)^{-(k+s_p+1)}\eta(p\tau)^{-4}.
\end{align*}
See \eqref{eq:defglx} for the definition of $g_{3,v/p}(p\tau)$ and Lemma \ref{lemma:L0Fmodular} for $F(\tau)$. We have seen in the proof of Theorem \ref{mainN} that there is a $v\in\{0,1,\dots,p-1\}$ and a sequence $c, c_1,c_2,\dots,c_{p-1}\in\C$ ($c_j\neq0$ for $1\leq j\leq p-1$) such that $\sum_{j=1}^{p-1}c_jf_j+cg_v$ is holomorphic on $\uhp$ and
\begin{equation}
\label{eq:fDecompfjgvh}
f=\left(\sum_{j=1}^{p-1}c_jf_j+cg_v\right)\cdot h
\end{equation}
where we have used \eqref{eq:UpkUpkprime} to express $U_{p,k}$ in terms of $U'_{p,k}$. Thus, we need to estimate the orders of $f_j$, $g_v$ and $h$ at cusps. It is known that $\ord_{a/c}\eta(p\tau)=\frac{(p,c)^2}{24p}$. This fact together with \eqref{eq:ordGEta} gives
\begin{equation}
\label{eq:ordpch}
\ord_{p/c}h=0,\quad 1\leq c\leq\frac{p-1}{2}.
\end{equation}
On the other hand, for $1\leq a\leq \frac{p-1}{2}$,
\begin{multline*}
\ord_{\frac{a}{p}}h=p\left(\left\{\frac{a}{p}\right\}^2-\left\{\frac{a}{p}\right\}\right)+\frac{p}{2}(k+s_p+1)\\
\cdot\left(\left\{\frac{a(p-1)}{2p}\right\}^2-\left\{\frac{a(p-1)}{2p}\right\}-\left\{\frac{a(p-3)}{2p}\right\}^2+\left\{\frac{a(p-3)}{2p}\right\}\right)
\end{multline*}
where $\{x\}:=x-[x]$. It follows that
\begin{equation}
\label{eq:ordaph}
\ord_{\frac{a}{p}}h>-(k+s_p+3)p/8.
\end{equation}
Let us now consider $g_v$. By Corollary \ref{cor:glxOrder}, we have
\begin{equation}
\label{eq:ordpcgv}
\ord_{\frac{p}{c}}g_v=0,\quad \ord_{\frac{a}{p}}g_v\geq0
\end{equation}
where $v>0$, $1\leq a,c\leq\frac{p-1}{2}$. For $g_0$, expanding the definitions (i.e. Lemma \ref{lemma:L0Fmodular}, \eqref{a3m} and \eqref{eq:AppellLerch}) we find
$$
F(\tau)=-\frac{7}{12}+\sum_{n\geq1}c_nq^{n/8}+\text{ ``nonholomorphic part''}
$$
from which $\ord_{i\infty} F=0$. It follows that \eqref{eq:ordpcgv} still holds for $v=0$. For $f_j$, Lemma \ref{lemma:ordUpkprime} implies
\begin{equation*}
\ord_{\frac{p}{c}}U'_{p,k}\left(\frac{\eta(p\tau)}{\eta(\tau)}\partial\widehat{A}_3\left(-\frac{j}{p}\right)\right)\geq\min_{0\leq m<p}\frac{(m,p)^2}{p}\cdot\ord_{\frac{p+cm}{cp}}\frac{\eta(p\tau)}{\eta(\tau)}\partial\widehat{A}_3\left(-\frac{j}{p}\right).
\end{equation*}
If $m=0$, then
$$
\frac{(m,p)^2}{p}\cdot\ord_{\frac{p+cm}{cp}}\frac{\eta(p\tau)}{\eta(\tau)}\partial\widehat{A}_3\left(-\frac{j}{p}\right)=-\frac{p-1}{24}+p\cdot\ord_{\frac{1}{c}}\partial\widehat{A}_3\left(-\frac{j}{p}\right)>-\frac{p-1}{24}
$$
according to Remark \ref{rema:ordUpkprime} and Corollary \ref{coro:orderPartialAlhat}. Otherwise if $m>0$, then
$$
\frac{(m,p)^2}{p}\cdot\ord_{\frac{p+cm}{cp}}\frac{\eta(p\tau)}{\eta(\tau)}\partial\widehat{A}_3\left(-\frac{j}{p}\right)=\frac{1}{p}\cdot\ord_{\frac{cm}{p}}\frac{\eta(p\tau)}{\eta(\tau)}\partial\widehat{A}_3\left(-\frac{j}{p}\right)=\frac{p-1}{24p}
$$
since $p\cdot\frac{j}{p}\in\Z$ and hence $\ord_{\frac{cm}{p}}\partial\widehat{A}_3\left(-\frac{j}{p}\right)=0$. Therefore
$$
\ord_{\frac{p}{c}}U'_{p,k}\left(\frac{\eta(p\tau)}{\eta(\tau)}\partial\widehat{A}_3\left(-\frac{j}{p}\right)\right)>-\frac{p-1}{24},
$$
and hence by Proposition \ref{prop:productHolo}
\begin{equation}
\label{eq:ordpcfj}
\ord_{\frac{p}{c}}f_j>\frac{1-p}{24p}-\frac{p-1}{24}=-\frac{s_p}{p},\quad1\leq j\leq p-1,\quad1\leq c\leq\frac{p-1}{2}.
\end{equation}
Next we calculate $\ord_{\frac{a}{p}}f_j$ for $1\leq a\leq\frac{p-1}{2}$. Again using Lemma \ref{lemma:ordUpkprime} we find that
\begin{align*}
\ord_{\frac{a}{p}}U'_{p,k}\left(\frac{\eta(p\tau)}{\eta(\tau)}\partial\widehat{A}_3\left(-\frac{j}{p}\right)\right)&\geq\min_{0\leq m<p}\frac{1}{p}\cdot\ord_{\frac{a+pm}{p^2}}\frac{\eta(p\tau)}{\eta(\tau)}\partial\widehat{A}_3\left(-\frac{j}{p}\right)\\
&=\frac{1}{p}\ord_{\frac{a}{p}}\frac{\eta(p\tau)}{\eta(\tau)}\partial\widehat{A}_3\left(-\frac{j}{p}\right)=\frac{p-1}{24p}.
\end{align*}
Thus
\begin{equation}
\label{eq:ordapfj}
\ord_{\frac{a}{p}}f_j\geq\frac{p-1}{24}+\frac{p-1}{24p}=\frac{s_p}{p},\quad1\leq j\leq p-1,\quad1\leq a\leq\frac{p-1}{2}.
\end{equation}
It follows from \eqref{eq:fDecompfjgvh}, \eqref{eq:ordpch}, \eqref{eq:ordpcgv}, \eqref{eq:ordpcfj} and Corollary \ref{coro:Hpart} that $\ord_{\frac{p}{c}}f>-\frac{s_p}{p}$. Since $\mathrm{div}_{\frac{p}{c}}f=p\cdot\ord_{\frac{p}{c}}f$ the first desired formula \eqref{eq:divfpc} follows. Similarly, the second desired formula \eqref{eq:divfap} follows from \eqref{eq:fDecompfjgvh}, \eqref{eq:ordaph}, \eqref{eq:ordpcgv}, \eqref{eq:ordapfj}, Corollary \ref{coro:Hpart} and the fact $\mathrm{div}_{\frac{a}{p}}f=\ord_{\frac{a}{p}}f$. Finally, the third desired formula \eqref{eq:divftau} holds since $f$ is holomorphic on $\uhp$.
\end{proof}

The above proposition allows us to check an identity of the form $f=\overline{f}$, where $f$ is as above and $\overline{f}$ is any modular function (with the trivial multiplier system) of the group $\Gamma_1(p)$ that is holomorphic on $\uhp$, with the aid of Frye and Garvan's MAPLE package \cite{gamp}. First we shall estimate the orders of $\overline{f}$ at all cusps except $1/p$: Assume that we have the lower bounds
\begin{equation}
\label{eq:divoverlinef}
\mathrm{div}_{\frac{p}{c}}\overline{f}\geq -e_{\frac{p}{c}},\quad\mathrm{div}_{\frac{a}{p}}\overline{f}\geq -e_{\frac{a}{p}},\quad1\leq a,\,c\leq\frac{p-1}{2},\,a\neq1
\end{equation}
where $e_{\frac{p}{c}},\,e_{\frac{a}{p}}\in\Q$.
\begin{cor}
\label{cor:divNmodular}
Suppose
\begin{equation*}
f=\sum_{n_0\leq n\in\Z}a_nq^n,\quad \overline{f}=\sum_{n_0\leq n\in\Z}b_nq^n.
\end{equation*}
If $a_n=b_n$ for all $n_0\leq n \leq n_1$ where
\begin{equation*}
n_1=\sum_{c=1}^{(p-1)/2}\max\{s_p,e_{\frac{p}{c}}\}+\sum_{a=2}^{(p-1)/2}\max\left\{\frac{(k+s_p+3)p}{8},e_{\frac{a}{p}}\right\},
\end{equation*}
then $f=\overline{f}$.
\end{cor}
\begin{proof}
Assume by contradiction that $f\neq\overline{f}$. According to the assumption \eqref{eq:divoverlinef} and Proposition \ref{prop:divNmodular} we have
\begin{align*}
\mathrm{div}_{\frac{p}{c}}(f-\overline{f})&\geq\min\{-s_p,-e_{\frac{p}{c}}\},\quad&\text{for } &1\leq c\leq \frac{p-1}{2},\\
\mathrm{div}_{\frac{a}{p}}(f-\overline{f})&\geq\min\left\{-\frac{(k+s_p+3)p}{8},-e_{\frac{a}{p}}\right\},\quad&\text{for } &2\leq a\leq \frac{p-1}{2},\\
\mathrm{div}_{\tau}(f-\overline{f})&\geq0,\quad&\text{for } &\tau\in\uhp.
\end{align*}
Since $a_n=b_n$ for $n\leq n_1$, $\mathrm{div}_{\frac{1}{p}}(f-\overline{f})\geq [n_1]+1$. Applying the valence formula \eqref{eq:valence} to $f-\overline{f}$ we obtain
\begin{align*}
0&=\sum_{c=1}^{(p-1)/2}\mathrm{div}_{\frac{p}{c}}(f-\overline{f})+\sum_{a=2}^{(p-1)/2}\mathrm{div}_{\frac{a}{p}}(f-\overline{f})+\mathrm{div}_{\frac{1}{p}}(f-\overline{f})+\sum_{\tau\in\Gamma_1(p)\backslash\uhp}\mathrm{div}_{\tau}(f-\overline{f})\\
&\geq-\sum_{c=1}^{(p-1)/2}\max\{s_p,e_{\frac{p}{c}}\}-\sum_{a=2}^{(p-1)/2}\max\left\{\frac{(k+s_p+3)p}{8},e_{\frac{a}{p}}\right\}+([n_1]+1)+0\\
&>0,
\end{align*}
which is a contradiction. Therefore $f=\overline{f}$.
\end{proof}

\begin{remark}
Proposition \ref{prop:divNmodular} and Corollary \ref{cor:divNmodular} hold as well for $f$ being \eqref{eq:MpskModular} instead of \eqref{eq:Nmodular}. For the proof, we proceed just as above, but with all the occurrences of $\widehat{A}_3$ replaced by $\widehat{A}_1$ and $g_v(\tau)$ redefined by $g_{1,v/p}(p\tau)$. Since our estimates of orders of $f_j$ and $g_v$ are independent of the level $\ell$, the proofs above also work for $f$ being \eqref{eq:MpskModular}. Moreover, it is possible to obtain more accurate lower bounds of $\mathrm{div}_{\frac{p}{c}}f$ and $\mathrm{div}_{\frac{a}{p}}f$ where $f$ is \eqref{eq:MpskModular} or \eqref{eq:Nmodular} using the exact formula in Corollary \ref{coro:orderPartialAlhat}, but the ones we have given suffice.
\end{remark}

In Appendices \ref{apx:A} and \ref{apx:B} we represent \eqref{eq:Nmodular} and \eqref{eq:MpskModular} for $p=5$, $7$ as polynomials of $t$ and $t^{-1}$ where $t$, depending on $p$, is a generalized Dedekind eta quotient. To prove each of these identities, we let $f$ be the left-hand side, $\overline{f}$ be the right-hand side and apply Corollary \ref{cor:divNmodular} or its $M_\omega$-version. Thus we just need to check the Fourier coefficients of $f$ and $\overline{f}$ up to the $q^{n_1}$-term. Below we give an estimate of $n_1$.

For $N_5(s,k)$ in Appendix \ref{apx:A}, we have
$$
s_p=\frac{p^2-1}{24}=1,\qquad\frac{(k+s_p+3)p}{8}\leq\frac{5\cdot(4+1+3)}{8}=5.
$$
Note that $\overline{f}$ is a linear combination of $t^{-2}$, $t^{-1}$ and $1$, where
$$
t=q\frac{J_{5,1}^5}{J_{5,2}^5}=\frac{\eta_{5,1}^5}{\eta_{5,2}^5}.
$$
By \eqref{eq:ordGEta} we have for $a=1,2$ and $c=1,2$,
\begin{align*}
\mathrm{div}_{\frac{5}{c}}(\eta_{5,1})=5\cdot\ord_{\frac{5}{c}}(\eta_{5,1})=\frac{1}{12},\quad&\mathrm{div}_{\frac{a}{5}}(\eta_{5,1})=\ord_{\frac{a}{5}}(\eta_{5,1})=\frac{a^2}{10}-\frac{a}{2}+\frac{5}{12},\\
\mathrm{div}_{\frac{5}{c}}(\eta_{5,2})=5\cdot\ord_{\frac{5}{c}}(\eta_{5,2})=\frac{1}{12},\quad&\mathrm{div}_{\frac{a}{5}}(\eta_{5,2})=\ord_{\frac{a}{5}}(\eta_{5,2})=\frac{2a^2}{5}-a+\frac{5}{12}.
\end{align*}
Hence
$$
\mathrm{div}_{\frac{5}{1}}(t)=0,\quad\mathrm{div}_{\frac{5}{2}}(t)=0,\quad\mathrm{div}_{\frac{1}{5}}(t)=1,\quad\mathrm{div}_{\frac{2}{5}}(t)=-1.
$$
It follows that in \eqref{eq:divoverlinef} we can choose $e_{\frac{5}{1}}=e_{\frac{5}{2}}=e_{\frac{2}{5}}=0$ and hence
\begin{equation*}
n_1=\max\{s_5,e_{\frac{5}{1}}\}+\max\{s_5,e_{\frac{5}{2}}\}+\max\left\{\frac{5(k+s_5+3)}{8},e_{\frac{2}{5}}\right\}\leq7.
\end{equation*}
This means if the Fourier coefficients of $f$ and $\overline{f}$ coincide up to the $q^{7}$-term, then $f=\overline{f}$ for $N_5(s,k)$.

The proof for $M_5(s,k)$ is the same as for $N_5(s,k)$: we need to check the Fourier coefficients up to the $q^{7}$-term as well.

Finally, the proofs for identities in Appendix \ref{apx:B} that deal with $N_7(s,k)$ and $M_7(s,k)$ are similar so we omit the details. We have in this case
\begin{gather*}
\mathrm{div}_{\frac{7}{1}}(t)=0,\quad\mathrm{div}_{\frac{7}{2}}(t)=0,\quad\mathrm{div}_{\frac{7}{3}}(t)=0,\\
\mathrm{div}_{\frac{1}{7}}(t)=1,\quad\mathrm{div}_{\frac{2}{7}}(t)=0,\quad\mathrm{div}_{\frac{3}{7}}(t)=-1,\\
e_{\frac{7}{1}}=e_{\frac{7}{2}}=e_{\frac{7}{3}}=e_{\frac{2}{7}}=0,\quad e_{\frac{3}{7}}=4,\\
s_p=2,\quad\frac{(k+s_p+3)p}{8}\leq\frac{77}{8},\quad n_1\leq\frac{101}{4}.
\end{gather*}
Thus to prove each identity in Appendix \ref{apx:B}, one need just to check the Fourier coefficients of both sides up to the $q^{25}$-term.

\begin{appendix}
\section{Representations of $N_5(s,k)$ and $M_5(s,k)$}
\label{apx:A}
\begin{flalign*}
&G:=\frac{J_{5,2}}{J_{5,1}} \qquad P:=\frac{J_{5,1}^2}{qJ_5^5} \qquad t:=q\frac{J_{5,1}^5}{J_{5,2}^5}.&
\end{flalign*}
~
\begin{align*}
&P\cdot G^2\cdot N_5(1,0)=-\frac{3}{10}t^{-1}-\frac{1}{10}, \qquad &&P\cdot G^2\cdot N_5(2,0)=-\frac{1}{10}t^{-1}-\frac{17}{10},\\
&P\cdot G^3\cdot N_5(1,1)=-\frac{3}{10}t^{-1}-\frac{11}{10}, \qquad &&P\cdot G^3\cdot N_5(2,1)=-\frac{1}{10}t^{-1}+\frac{13}{10},\\
&P\cdot G^4\cdot N_5(1,2)=-\frac{11}{10}t^{-1}+\frac{3}{10}, \qquad &&P\cdot G^4\cdot N_5(2,2)=-\frac{1}{5}t^{-1}-\frac{2}{5},\\
&P\cdot G^5\cdot N_5(1,3)=\frac{1}{10}t^{-1}-\frac{3}{10}, \qquad &&P\cdot G^5\cdot N_5(2,3)=-\frac{9}{5}t^{-1}+\frac{2}{5},\\
&P\cdot G^6\cdot N_5(1,4)=-t^{-1}, \qquad &&P\cdot G^6\cdot N_5(2,4)=\frac{5}{12}t^{-2}+\frac{1}{2}t^{-1}+\frac{5}{12},\\
~\\
&P\cdot G^2\cdot M_5(1,0)=-\frac{3}{10}t^{-1}-\frac{1}{10}, \qquad &&P\cdot G^2\cdot M_5(2,0)=-\frac{1}{10}t^{-1}+\frac{33}{10},\\
&P\cdot G^3\cdot M_5(1,1)=-\frac{4}{5}t^{-1}+\frac{2}{5}, \qquad &&P\cdot G^3\cdot M_5(2,1)=-\frac{1}{10}t^{-1}+\frac{13}{10},\\
&P\cdot G^4\cdot M_5(1,2)=\frac{2}{5}t^{-1}+\frac{4}{5}, \qquad &&P\cdot G^4\cdot M_5(2,2)=-\frac{17}{10}t^{-1}-\frac{9}{10},\\
&P\cdot G^5\cdot M_5(1,3)=\frac{1}{10}t^{-1}-\frac{3}{10}, \qquad &&P\cdot G^5\cdot M_5(2,3)=\frac{16}{5}t^{-1}+\frac{2}{5},\\
&P\cdot G^6\cdot M_5(1,4)=4t^{-1}, \qquad &&P\cdot G^6\cdot M_5(2,4)=-2t^{-1}.
\end{align*}

\section{Representations of $N_7(s,k)$ and $M_7(s,k)$}
\label{apx:B}
\begin{flalign*}
&G:=\frac{J_{7,3}}{J_{7,2}} \qquad P:=\frac{J_{7,1}^2}{qJ_7^5} \qquad t:=q\frac{J_{7,1}^3}{J_{7,2}^2J_{7,3}}.&
\end{flalign*}
~
\begin{flalign*}
&P\cdot G^3\cdot N_7(1,0)=-\frac{6}{7}t^{-1}+\frac{10}{7}-t+\frac{3}{7}t^2&\\
&P\cdot G^4\cdot N_7(1,1)=-\frac{5}{14}t^{-1}-\frac{11}{7}+2t-\frac{1}{14}t^2,&\\
&P\cdot G^5\cdot N_7(1,2)=-\frac{17}{14}t^{-1}+\frac{6}{7}+3t-\frac{37}{14}t^2,&\\
&P\cdot G^6\cdot N_7(1,3)=-\frac{1}{14}t^{-1}+\frac{19}{14}-\frac{16}{7}t+\frac{11}{14}t^2+\frac{3}{14}t^3,&\\
&P\cdot G^7\cdot N_7(1,4)=-\frac{9}{7}t^{-1}+\frac{38}{7}-\frac{64}{7}t+\frac{50}{7}t^2-\frac{15}{7}t^3,&\\
&P\cdot G^8\cdot N_7(1,5)=-t^{-1}+1+3t-5t^2+2t^3,&\\
&P\cdot G^9\cdot N_7(1,6)=-\frac{27}{14}t^{-1}+\frac{67}{7}-\frac{243}{14}t+\frac{195}{14}t^2-\frac{31}{7}t^3+\frac{3}{14}t^4,&
\end{flalign*}
~
\begin{flalign*}
&P\cdot G^3\cdot N_7(2,0)=-\frac{3}{14}t^{-1}-\frac{22}{7}+5t-\frac{23}{14}t^2,&\\
&P\cdot G^4\cdot N_7(2,1)=-\frac{3}{14}t^{-1}+\frac{13}{7}-3t+\frac{19}{14}t^2,&\\
&P\cdot G^5\cdot N_7(2,2)=-\frac{3}{7}t^{-1}+\frac{5}{7}-t+\frac{5}{7}t^2,&\\
&P\cdot G^6\cdot N_7(2,3)=-\frac{15}{7}t^{-1}+\frac{40}{7}-\frac{39}{7}t+\frac{18}{7}t^2-\frac{4}{7}t^3,&\\
&P\cdot G^7\cdot N_7(2,4)=-\frac{1}{14}t^{-1}-\frac{23}{14}+\frac{40}{7}t-\frac{87}{14}t^2+\frac{31}{14}t^3,&\\
&P\cdot G^8\cdot N_7(2,5)=-2t^{-1}+9-15t+11t^2-3t^3,&\\
&P\cdot G^9\cdot N_7(2,6)=\frac{37}{14}t^{-1}-\frac{90}{7}+\frac{333}{14}t-\frac{289}{14}t^2+\frac{57}{7}t^3-\frac{15}{14}t^4,&
\end{flalign*}
~
\begin{flalign*}
&P\cdot G^3\cdot N_7(3,0)=-\frac{1}{14}t^{-1}+\frac{9}{7}-3t+\frac{25}{14}t^2,&\\
&P\cdot G^4\cdot N_7(3,1)=-\frac{1}{14}t^{-1}-\frac{12}{7}+6t-\frac{59}{14}t^2,&\\
&P\cdot G^5\cdot N_7(3,2)=-\frac{1}{7}t^{-1}+\frac{18}{7}-5t+\frac{18}{7}t^2,&\\
&P\cdot G^6\cdot N_7(3,3)=-\frac{3}{14}t^{-1}-\frac{41}{14}+\frac{50}{7}t-\frac{65}{14}t^2+\frac{9}{14}t^3,&\\
&P\cdot G^7\cdot N_7(3,4)=-\frac{20}{7}t^{-1}+\frac{79}{7}-\frac{115}{7}t+\frac{73}{7}t^2-\frac{17}{7}t^3,&\\
&P\cdot G^8\cdot N_7(3,5)=\frac{7}{12}t^{-2}-\frac{2}{3}t^{-1}-\frac{5}{3}+\frac{19}{6}t-t^2-t^3+\frac{7}{12}t^4,&\\
&P\cdot G^9\cdot N_7(3,6)=-\frac{23}{7}t^{-1}+\frac{96}{7}-\frac{158}{7}t+\frac{128}{7}t^2-\frac{51}{7}t^3+\frac{8}{14}t^4,&
\end{flalign*}
~
\begin{flalign*}
&P\cdot G^3\cdot M_7(1,0)=-\frac{5}{14}t^{-1}-\frac{4}{7}-t+\frac{27}{14}t^2,&\\
&P\cdot G^4\cdot M_7(1,1)=-\frac{6}{7}t^{-1}+\frac{3}{7}+2t-\frac{11}{7}t^2,&\\
&P\cdot G^5\cdot M_7(1,2)=\frac{2}{7}t^{-1}+\frac{13}{7}-4t+\frac{13}{7}t^2,&\\
&P\cdot G^6\cdot M_7(1,3)=-\frac{1}{14}t^{-1}+\frac{19}{14}-\frac{16}{7}t+\frac{11}{14}t^2+\frac{3}{14}t^3,&\\
&P\cdot G^7\cdot M_7(1,4)=\frac{3}{14}t^{-1}-\frac{29}{14}+\frac{27}{7}t-\frac{33}{14}t^2+\frac{5}{14}t^3,&\\
&P\cdot G^8\cdot M_7(1,5)=-t^{-1}+8-18t+16t^2-5t^3,&\\
&P\cdot G^9\cdot M_7(1,6)=\frac{71}{14}t^{-1}-\frac{129}{7}+\frac{345}{14}t-\frac{197}{14}t^2+\frac{18}{7}t^3+\frac{3}{14}t^4,&
\end{flalign*}
~
\begin{flalign*}
&P\cdot G^3\cdot M_7(2,0)=-\frac{3}{14}t^{-1}+\frac{27}{7}-2t-\frac{23}{14}t^2,&\\
&P\cdot G^4\cdot M_7(2,1)=-\frac{3}{14}t^{-1}+\frac{13}{7}-3t+\frac{19}{14}t^2,&\\
&P\cdot G^5\cdot M_7(2,2)=-\frac{27}{14}t^{-1}+\frac{47}{7}-8t+\frac{45}{14}t^2,&\\
&P\cdot G^6\cdot M_7(2,3)=\frac{5}{14}t^{-1}+\frac{3}{14}-\frac{18}{7}t+\frac{43}{14}t^2-\frac{15}{14}t^3,&\\
&P\cdot G^7\cdot M_7(2,4)=-\frac{11}{8}t^{-1}+\frac{41}{7}-\frac{51}{7}t+\frac{23}{7}t^2-\frac{2}{7}t^3,&\\
&P\cdot G^8\cdot M_7(2,5)=5t^{-1}-19+27t-17t^2+4t^3,&\\
&P\cdot G^9\cdot M_7(2,6)=-\frac{13}{7}t^{-1}+\frac{50}{7}-\frac{68}{7}t+\frac{34}{7}t^2+\frac{1}{7}t^3-\frac{4}{7}t^4,&
\end{flalign*}
~
\begin{flalign*}
&P\cdot G^3\cdot M_7(3,0)=-\frac{1}{14}t^{-1}+\frac{9}{7}-3t+\frac{25}{14}t^2,&\\
&P\cdot G^4\cdot M_7(3,1)=-\frac{1}{14}t^{-1}+\frac{37}{7}-8t+\frac{39}{14}t^2,&\\
&P\cdot G^5\cdot M_7(3,2)=-\frac{1}{7}t^{-1}-\frac{31}{7}+9t-\frac{31}{7}t^2,&\\
&P\cdot G^6\cdot M_7(3,3)=-\frac{19}{7}t^{-1}+\frac{67}{7}-\frac{69}{7}t+\frac{13}{7}t^2+\frac{8}{7}t^3,&\\
&P\cdot G^7\cdot M_7(3,4)=\frac{29}{7}t^{-1}-\frac{68}{7}+\frac{32}{7}t+\frac{24}{7}t^2-\frac{17}{7}t^3,&\\
&P\cdot G^8\cdot M_7(3,5)=-3t^{-1}+10-12t+6t^2-t^3,&\\
&P\cdot G^9\cdot M_7(3,6)=\frac{17}{14}t^{-1}-\frac{44}{7}+\frac{153}{7}t-\frac{101}{14}t^2+\frac{5}{7}t^3-\frac{9}{14}t^4.&
\end{flalign*}
\end{appendix}

\bibliographystyle{plain}
\bibliography{ref}


%
%









\end{document}